\documentclass[a4paper,reqno,12pt]{amsart}
\usepackage{amsmath}
\usepackage{amssymb}
\usepackage{amsthm}
\usepackage[noadjust]{cite}
\usepackage{mathrsfs}
\usepackage[all,cmtip]{xy}
\usepackage{cite}
\usepackage{cases}
\usepackage{enumitem}
\usepackage{indentfirst}
\usepackage{hyperref}
\usepackage{comment}

\usepackage{geometry}
\usepackage{xy}
\allowdisplaybreaks[4]

\newtheorem{thm}{Theorem}[section]
\newtheorem{prop}[thm]{Proposition}

\newtheorem{lem}[thm]{Lemma}

\newtheorem{cor}[thm]{Corollary}

\theoremstyle{definition}
\newtheorem{definition}[thm]{Definition}
\newtheorem{example}[thm]{Example}

\theoremstyle{remark}
\newtheorem{remark}[thm]{Remark}

\numberwithin{equation}{section}

\newcommand{\bN}{\mathbb{N}}
\newcommand{\bP}{\mathbb{P}}
\newcommand{\bQ}{\mathbb{Q}}
\newcommand{\bR}{\mathbb{R}}
\newcommand{\bZ}{\mathbb{Z}}

\newcommand\cE{\mathcal{E}}
\newcommand\cF{{\mathcal{F}}}
\newcommand\cG{{\mathcal{G}}}

\newcommand\cI{{\mathcal{I}}}

\newcommand\cO{{\mathcal{O}}}
\newcommand\cP{\mathcal{P}}
\newcommand\cQ{{\mathcal{Q}}}

\newcommand{\rounddown}[1]{\lfloor{#1}\rfloor}
\newcommand{\roundup}[1]{\lceil{#1}\rceil}

\newcommand{\ff}{{f^{\prime}}}
\newcommand{\BB}{{B^{\prime}}}

\newcommand{\DD}{{D^{\prime}}}
\newcommand{\FF}{{F^{\prime}}}

\newcommand{\MM}{{M^{\prime}}}

\newcommand{\UU}{{U^{\prime}}}

\newcommand{\XX}{{X^{\prime}}}
\newcommand{\YY}{{Y^{\prime}}}
\newcommand{\ZZ}{{Z^{\prime}}}

\newcommand{\ZZZ}{Z^{\prime \prime}}

\newcommand{\ZZZZ}{Z^{\prime \prime \prime}}

\newcommand{\MMi}{{M^{\prime}_i}}

\newcommand{\mult}{\operatorname{mult}}
\newcommand{\Nklt}{\operatorname{Nklt}}

\newcommand{\Supp}{\operatorname{Supp}}

\newcommand{\vol}{\operatorname{vol}}
\newcommand{\Ivol}{\operatorname{Ivol}}
\newcommand{\Div}{\operatorname{Div}}
\newcommand{\Nlc}{\operatorname{Nlc}}

\hypersetup{
	colorlinks=true,
	citecolor=red,
	linkcolor=blue,
	urlcolor=green,
}

\begin{document}
\title{ON THE DCC PROPERTY OF IITAKA VOLUME WITH REAL COEFFICIENTS AND GENERALISED PAIRS}
\date{\today}

\author{PINXIAN BIE}
\address{PINXIAN BIE, School of Mathematical Sciences, Fudan University, Shanghai, 200433, China}
\email{23110180002@m.fudan.edu.cn}

\begin{abstract}
We investigate the DCC property of the set of Iitaka volumes of a given set of pairs of varieties. We both generalize previous results of Birkar and Li about usual pairs to the real coefficient case, and also establish similar results on generalised pairs, where some natural boundedness assumptions are required for technical reasons.
\end{abstract}

\maketitle
\pagestyle{myheadings} \markboth{\hfill P.X.~Bie
	\hfill}{\hfill On the DCC property of Iitaka volume with real coefficients and generalised pairs\hfill}

\tableofcontents

\section{Introduction}
We work over an algebraically closed field $k$ of characteristic zero.
~\\

Boundedness properties of a given set of algebraic varieties are extensively studied in recent years. Among all boundedness questions, a basic question is about the DCC property of the Iitaka volume of the log canonical divisor $K_X+B$ of a given set of pairs. To be more precise, if we are given a set of pairs
$$\mathcal{I}_{lc}(d, \Phi)=\{(X,B)|(X,B) \text {is projective lc}, \mathrm{dim}X=d, B\in \Phi\}$$
where $B\in \Phi$ means that the coefficients of $B$ lie in a fixed DCC set $\Phi$. Then we ask whether the corresponding set of Iitaka volume
$$ \{\mathrm{Ivol}(K_X+B)|(X,B)\in \mathcal{I}_{lc}(d, \Phi)\}$$
is also a DCC set depending only on $d$ and $\Phi$.

When $K_X+B$ is big, the Iitaka volume is just the usual volume and we may denote it by $\vol(K_X+B)$. An important known fact is that if $\Phi$ satisfies the DCC, the set of usual volumes also satisfies DCC (\cite{ACCLCT}). It's natural to ask whether similar phenomena occurs for the intermediate Kodaira dimension case, and many partial results have been established in recent years. 

Zhan Li first proved a special case in \cite{Iitakavolume} when $(X,B)$ has $\epsilon$-lc singularities and the boundary $B$ is a big $\bQ$-divisor over $Z$, where $X\rightarrow Z$ is a contraction with $K_X+B\sim_{\bQ}0/Z$. Birkar proved a more general version in \cite{BVGP}, where instead of the Fano type assumption on the fibration, he showed the DCC for the set of Iitaka volumes of lc-trivial fibrations with a given ample $\bZ$-divisor $A$ of fixed volume on the general fiber. This additional condition is natural and it implies boundedness of the general fiber. Chen, Han and Liu proved the DCC property in dimension $d\leq 3$ without any extra assumptions in \cite{CHL24}, where their proof relied on the existence of good minimal model and existence of $n$-complements in low dimensions.

Most of the above results are restricted to $\bQ$-divisors and the extension to $\bR$-divisors is not trivial. For $\bR$-divisors, the Iitaka volume doesn't behave well enough, and we introduce the invariant Iitaka volume instead. It's expected that invariant Iitaka volume also has the DCC property. In this paper we shall prove that under suitable conditions, the DCC property of invariant Iitaka volume indeed holds for $\bR$-divisors.

\begin{definition}\label{DEF1}(\cite{BVGP})
Let $d \in \mathbb{N}$, $\Phi \subset \mathbb{R}^{\geq 0}$, and $u \in \mathbb{R}^{\geq 0}$. Let $\mathcal{I}_{lc}(d, \Phi, u)$ be the set of projective pairs $(X, B)$ such that
\begin{itemize}
    \item $(X, B)$ is lc of dimension $d$,
    \item the coefficients of $B$ are in $\Phi$,
    \item $f: X \to Z$ is a contraction with $K_X + B \sim_{\mathbb{R}} 0/Z$,
    \item $\kappa(K_X + B) = \dim Z$, and
    \item there is an effective $\bR$-divisor $A \geq 0$ on $X$ such that $A\in \Phi$  and over some non-empty open subset of $Z$: $(X, B + tA)$ is lc for some $t > 0$ and $A$ is ample,
    \item $\operatorname{vol}(A|_F) = u$ for the general fibres $F$ of $f$.
\end{itemize}
Define $\mathcal{I}_{lc}(d, \Phi, <u)$ similarly by replacing the condition $\operatorname{vol}(A|_F) = u$ with $\operatorname{vol}(A|_F) < u$. And define $\mathcal{I}_{klt}(d, \Phi, u)$ and $\mathcal{I}_{klt}(d, \Phi, <u)$ similarly by replacing the lc condition of $(X, B)$ with klt.

\end{definition}

\begin{thm}\label{MAIN THM1}
  Let $d\in \bN$, $\Phi \subset \bR^{\geq0}$ be a DCC set, and $u\in \bR^{>0}$. Then the sets of invariant Iitaka volumes 
  $$\{\mathrm{Ivol}_{\iota}(K_X+B)|(X,B)\in \mathcal{I}_{lc}(d, \Phi, u)\}$$
  and
  $$\{\mathrm{Ivol}_{\iota}(K_X+B)|(X,B)\in \mathcal{I}_{klt}(d, \Phi, <u)\}$$
  satisfy the DCC property.
\end{thm}

Another natural question is to ask if DCC property of Iitaka volume holds for generalised pairs $(X, B+M)$ as well. To be more precise, if the coefficients of $B$ belong to a fixed DCC set $\Phi$ and the nef part $M'=\sum \mu_jM'_j$ where $M'_j$ are nef Cartier and $\mu_j\in \Phi$, then we are asked whether the set $\{\mathrm{Ivol}(K_X+B+M)|(X,B+M) \in \mathcal{I}_{glc}(d,\Phi)\} $ is DCC or not. Birkar first proved the case when $(X, B+M)$ is g-lc and $K_X+B+M$ is big in \cite{BVGP}. In general, the generalised pairs case is much more subtle. In fact, Birkar and Hacon showed in \cite{VGP} that without any extra assumptions the DCC of Iitaka volume of generalised pairs may fail. So it's necessary to impose some additional conditions in this case. One may wish to give similar assumption as Birkar did in the usual pair case to ensure boundedness of the general fibers and follow through the original proof in \cite{BVGP}, but there are still some issues and we can only obtain some weaker results here.

\begin{definition}\label{DEF2}
    Let $d,q \in \mathbb{N}$, $\Phi \subset \mathbb{Q}^{\geq 0}$, and $u \in \mathbb{Q}^{\geq 0}$. Let $\mathcal{I}_{glc}(d, \Phi,q, u)$ be the set of projective generalised pairs $(X, B+M)$ with data $\phi:X'\rightarrow X$ and $M'$ such that
\begin{itemize}
    \item $(X, B+M)$ is generalised lc of dimension $d$,
    \item the coefficients of $B$ are in $\Phi$,
    \item $qM'$ is nef Cartier,
    \item $f: X \to Z$ is a contraction with $K_X + B +M \sim_{\mathbb{Q}} 0/Z$,
    \item $\kappa(K_X + B +M) = \dim Z$, 
    \item we have a generalised adjunction formula
    $$q(K_X+B+M)\sim qf^*(K_Z+B_Z+M_Z), $$ and
    \item there is an effective $\bQ$-divisor $A \geq 0$ on $X$ such that $A\in \Phi$ and over some non-empty open subset of $Z$: $(X, B + tA +M)$ is generalised lc for some $t > 0$ and $A$ is ample,
    \item $\operatorname{vol}(A|_F) = u$ for the general fibres $F$ of $f$.
\end{itemize}
Define $\mathcal{I}_{glc}(d, \Phi, q, <u)$ similarly by replacing the condition $\operatorname{vol}(A|_F) = u$ with $\operatorname{vol}(A|_F) < u$. And define $\mathcal{I}_{gklt}(d, \Phi,q, u)$ and $\mathcal{I}_{gklt}(d, \Phi,q, <u)$ similarly by replacing the glc condition of $(X, B+M)$ with gklt. When $\Phi \subset \bR^{\geq0}$, we define the above sets in a similar way.

Define $\mathcal{I}_{glc}^{ft}(d, q,\Phi)$ by replacing the last three conditions in $\mathcal{I}_{glc}(d, \Phi, q, u)$ by 
\begin{itemize}
    \item $X$ is of Fano type over $Z$, i.e. $-K_X$ is big over $Z$.
\end{itemize}
Define $\mathcal{I}_{gklt}^{ft}(d, \Phi)$ similarly by replacing the glc condition of $(X, B+M)$ with gklt.
\end{definition}

\begin{thm}\label{MAIN THM2}
  Let $d,q\in \bN$, $\Phi \subset \bQ^{\geq0}$ be a DCC set, and $u\in \bQ^{>0}$. Then the set
  $$\{\mathrm{Ivol}(K_X+B+M)|(X,B)\in \mathcal{I}_{glc}(d, \Phi,q, <u)\}$$
  satisfy the DCC property.  
\end{thm}

An important corollary is the following result regarding Fano type g-klt-trivial fibrations.

\begin{cor}\label{MAIN COR1}
  Let $d,q\in \bN$ and $\Phi \subset \bQ^{\geq0}$ be a DCC set. Then the set
  $$\{\mathrm{Ivol}(K_X+B+M)|(X,B+M)\in \mathcal{I}_{gklt}^{ft}(d,q, \Phi)\}$$
  satisfies the DCC property.
\end{cor}

As the usual pair case, the above result about generalised pairs also holds for $\bR$-divisors.

\begin{cor}\label{MAIN COR2}
  Let $d\in \bN$, $\Phi \subset \bR^{\geq0}$ a DCC set, and  $\Phi' \subset \bR^{\geq0}$ a finite set. Then the set of invariant Iitaka volumes
  $$\{\mathrm{Ivol}_{\iota}(K_X+B+M)|(X,B+M)\in \mathcal{I}_{gklt}^{ft}(d, \Phi,\Phi')\}$$
  satisfies the DCC property. Here in $\mathcal{I}_{gklt}^{ft}(d, \Phi,\Phi')$, we replace the condition on $qM'$ being Cartier by the following:
  \begin{itemize}
      \item $M'=\sum\mu_jM_j'$ where $\mu_j\in \Phi'$ and $M_j'$ is Cartier nef for any $j$.
  \end{itemize}
\end{cor}

On the other hand, if we restrict ourselves to generalised klt pairs $(X,B+M)$, we may also omit the assumption on $q(K_X+B+M)\sim qf^*(K_Z+B_Z+M_Z)$. In particular, we have the following result, where the proof is quite different from that of Theorem \ref{MAIN THM2}. We also remark that, unlike the setting in Theorem \ref{MAIN THM2}, the theorem below only require $A$ to be an integral divisor, so that $A$ may not be effective.

\begin{thm}\label{MAINTHMnew}
 Let $d,q$ be two positive integers, $u\in \bQ^{>0}$ a positive rational number, and $\Phi\subset \bQ^{\geq0}$ a DCC set. Consider the set of $(X,B+M),A$ such that
 \begin{itemize}    
 \item $(X,B+M)$ is a projective generalised klt pair of dimension $d$ with data $\pi: \XX \rightarrow X$ and nef part $\MM$,    
 \item $B\in \Phi$ and $q\MM$ is nef Cartier,    
 \item there is a contraction $f:X\to Z$ such that $K_X+B+M\sim_\bQ 0/Z$, 
 \item $\kappa(K_X+B+M)=\dim Z$, and
 \item there is an integral divisor $A$ on $X$ such that $0<\vol (A|_F)\leq u$, where $F$ is the general fiber of $f:X\rightarrow Z$.
 \end{itemize}  

 Then the set of Iitaka volumes $\Ivol(K_X+B+M)$ of such $(X,B+M),A$ belongs to a fixed DCC set depending only on $d,q,u,\Phi$.
\end{thm}

The Iitaka volume of anti-log canonical divisors $-(K_X+B)$ behaves very differently, and many counter examples of generalised pairs are constructed by setting $-(K_X+B)$ to be semi-ample, $M\sim_{\bQ}-2(K_X+B)$ so that $K_X+B+M\sim_{\bQ}-(K_X+B)$ and the DCC property of Iitaka volume fails. On the other hand, under some additional assumptions on the singularities, the Iitaka volume of $-(K_X+B)$ behaves well enough.

\begin{cor}\label{MAIN COR3}
  Let $d,l$ be two positive integers, $\epsilon$, $v$ two positive rational numbers, and $\Phi\subset \bQ^{\geq0}$ is a finite set.

  Consider projective pairs $(X,B)$ such that 
  \begin{itemize}
      \item $(X,B)$ is $\epsilon$-lc of dimension $d$ with $B\in \Phi$,
      \item $-(K_X+B)$ is a semi-ample $\bQ$-divisor that defines a contraction $f:X\rightarrow Z$,
      \item $l(K_X+B)$ is Cartier, and
      \item there is a $\bQ$-Cartier $\bZ$-divisor $A$ on $X$ with $\mathrm{vol}(A|_F)=v$ for the general fibers $F$ of $f$.
  \end{itemize}

  Then there exists a finite set $J$ depending only on $d,l, \epsilon, v$, and $\Phi$ such that the Iitaka volume of $-(K_X+B)$ belongs to $J$. 
\end{cor}

We also provide two boundedness results about g-klt fibration with fixed Iitaka volume $\mathrm{Ivol}(K_X+B+M)$. The proof is similar to that of Theorem \ref{MAIN THM2}.

\begin{thm}\label{MAINTHM3} 
Let $d,q\in \bN$, $u,v\in \bQ^{>0}$, and $\Phi\subset \bQ^{\geq 0}$ be a DCC set. Consider the set of $(X,B+M),A$ such that
\begin{itemize}
    \item $(X,B+M)$ is a projective generalised klt pair of dimension $d$ with data $\pi: \XX \rightarrow X$ and nef part $\MM$,
    \item $B\in \Phi$ and $q\MM$ is nef Cartier,
    \item $K_X+B+M$ is semi-ample defining a contraction $f:X\rightarrow Z$,
    \item $\mathrm{Ivol}(K_X+B+M)=v$,
    \item we have an adjunction formula
    $$ q(K_X+B+M)\sim qf^*(K_Z+B_Z+M_Z)$$, and
    \item there is an effective $\bQ$-divisor $A\geq 0$ on $X$ that is ample such that $A\in \Phi$ and $\vol (A|_F)\leq u$, where $F$ is the general fiber of $f:X\rightarrow Z$.
\end{itemize}

Then there is a bounded family of g-klt pairs $\cP$, a positive integer $l\in \bN$, and a positive number $\epsilon \in \bQ^{>0}$ depending only on $d,u,v,q,\Phi$ such that $(Z,B_Z+M_Z)$ belongs to $\cP$, $l(K_Z+B_Z+M_Z)$ is very ample, and $(X, B+M)$ is generalised $\epsilon$-lc.
\end{thm}

\begin{cor}\label{MAINCOR4}
  Let $d,q\in \bN$, $v\in \bQ^{>0}$, and $\Phi\subset \bQ^{\geq 0}$ be a DCC set. Consider the set of $(X,B+M)$ such that
  \begin{itemize}    
  \item $(X,B+M)$ is a projective generalised klt pair of dimension $d$ with data $\pi: \XX \rightarrow X$ and nef part $\MM$,    
  \item $B\in \Phi$ and $q\MM$ is nef Cartier,   
  \item $K_X+B+M$ is semi-ample defining a contraction $f:X\rightarrow Z$,    
  \item $\mathrm{Ivol}(K_X+B+M)=v$, and   
  \item $X$ is of Fano type over $Z$.
  \end{itemize}
  
Then there is a bounded family of g-klt pairs $\cP$, a positive integer $l\in \bN$, and a positive number $\epsilon \in \bQ^{>0}$ depending only on $d,q,v,\Phi$ such that $(Z,B_Z+M_Z)$ belongs to $\cP$, $l(K_Z+B_Z+M_Z)$ is very ample, and $(X, B+M)$ is generalised $\epsilon$-lc. Moreover, the set of such $X$ is also bounded.
\end{cor}

In the end of this paper, we provide several examples showing that some assumptions in \ref{DEF2} cannot be removed. These examples also reflect the complexity of g-lc trivial fibrations. On the other hand, we also give a short proof showing that at least in dimension two, Theorem \ref{MAIN THM2} can be improved.

~\\

\paragraph{\textbf{Sketch of proofs}} 
We start with Theorem \ref{MAIN THM1}. Let $(X,B),A\in \cI_{lc}(d,\Phi,<u)$. Let $f:X\to Z$ be the contraction defined by the semi-ample $\bR$-divisor $K_X+B$. Recall that in \cite[Lemma 7.4]{BVGP}, Birkar showed that in the $\bQ$-divisor case, there exists positive integers $p, q$ depending only on $d,\Phi$, and $u$ such that $q(K_X+B)\sim qf^*(K_Z+B_Z+M_Z)$ and $pM_{Z'}$ is nef Cartier. This helps us to control the coefficients of the moduli part effectively and thus the DCC of Iitaka volume holds by \cite[Theorem 1.3]{BVGP}. However in the $\bR$-divisor case such integers may not exist. To address this issue, we need to decompose the $\bR$-divisor log pairs $K_X+B$ into a $\bR$-linear sum of $\bQ$-divisors, and this decomposition must be uniform. That is, we may write $K_X+B = \sum r_i(K_X+B_i)$ so that $\sum r_i =1$, $K_X+B_i\sim_{\bQ}0/Z$ for all $i$, $(X,B_i)$ is lc for all $i$, and $r_i\in \Psi$ for a fixed DCC set $\Psi$. To do this, we first use the uniform decomposition of adjunction formula introduced in \cite[Theorem 3.3]{effectiveadjunctionformulawithrealcoefficients}, where the authors only studied the finite coefficient case. Then we use some standard MMP tools to reduce the DCC case to the finite case.  

We note that Zhu also uses similar method in \cite{Zhu25}, where the author is interested in different aspects.
~\\

Now consider the generalised pairs case, where the situation is much more difficult than the usual pair case. We still have a generalised version of the canonical bundle formula: $K_X+B+M\sim_{\bQ}f^*(K_Z+B_Z+M_Z)$ and the coefficients of $B_Z$ still belong to a fixed DCC set, but controlling the moduli part $M_Z$ is much more difficult. The main difficulty is that we cannot control the torsion index of $K_X+B+M$ along the general fiber $F$ of $f:X\rightarrow Z$, even if we fix $A$ with $\mathrm{vol}(A|_F)=u$. This is because that although the general fibers $F$ still belong to a bounded family, we don't have a uniform $q$ to ensure $q(K_F+B_F+M_F)\sim 0$. For example, if $F$ is a fixed elliptic curve, $K_F=B_F=0$, and $M_F$ is a torsion divisor on $F$ with arbitrarily large torsion index, then such $q$ certainly doesn't exist.

One option to solve this issue is to simply assume $f:X\rightarrow Z$ is of Fano type. This condition ensures that any Cartier divisor $D\sim_{\bQ} 0$ on $F$ is actually linearly equivalent to $0$. We note that in our proof of Corollary \ref{MAIN COR1}, the Fano type condition is only used here.
So it seems more natural to give a weaker assumption, and our main observation is that assuming $q(K_X+B+M)\sim qf^*(K_Z+B_Z+M_Z)$ for some fixed integer $q$ is enough.

To prove Theorem \ref{MAIN THM2} and Corollary \ref{MAIN COR1}, we first generalize several known boundedness type results of usual pairs to the g-pairs setting, like Theorem \ref{g-pair version of 6.4}, which is the g-pair version of \cite[Theorem 6.4]{GOPV}. Such generalizations may be useful elsewhere. The proof of Theorem \ref{g-pair version of 6.4} is a key part of the argument, where we use methods similar to \cite[Lemma 4.4]{BABI}, but we have to control the moduli part as well, which requires additional works. We then follow the ideas of \cite{BVGP} and use our new combined results for g-pairs to prove an effective adjunction formula (Theorem \ref{g-pair effective canonical bundle formula}) and use it to prove Theorem \ref{MAIN THM2}. The proof of Corollary \ref{MAIN COR2} is similar to Theorem \ref{MAIN THM1} and we omit it here.

Another option is to assume $(X,B+M)$ is generalised klt to begin with, and in this direction we have Theorem \ref{MAINTHMnew}. The basic idea is, although we still cannot find $p$ such that $p(K_X+B+M)\sim pf^*(K_Z+B_Z+M_Z)$, there is an alternative way to control $M_Z$. To do this, we use MMP theory to construct a new generalised pair $K_{X''}+tA''$ that is semi-ample but not big, and then obtain a non-birational contraction $X''\to Y/Z$. The main observation is that $X''\to Y$ is of Fano type, and we may use the effective adjunction formula for Fano type fibrations (Lemma \ref{g-pair Fanotype eff adjunction}) and induction to prove the result.

Corollary \ref{MAIN COR3} is a direct consequence of Corollary \ref{MAIN THM2} and \cite[Theorem 1.8]{CHL24}. Here we may take $M=-2(K_X+B)$ so that $K_X+B+M = -(K_X+B)$. The boundedness results Theorem \ref{MAINTHM3} and Corollary \ref{MAINCOR4} are direct application of Theorem \ref{g-pair effective canonical bundle formula}, \cite[Theorem 1.4]{BVGP} and \cite[Theorem 1.2]{FTF}.
 ~\\

\paragraph{\textbf{Acknowledgements}} The author expresses his gratitude to his advisor Meng Chen for great support and encouragement. He would like to thank Jingjun Han, Wenfei Liu and Minzhe Zhu for relevant discussions. He also thanks Sicheng Ding and Zhengjie Yu for helpful suggestions.

\section{Preliminaries}

\subsection{Divisors}
\begin{definition}[ACC sets and DCC sets]
	Let $\Phi\subset \bR^{\geq0}$. We say $\Phi$ satisfies the \textit{ascending chain condition} (ACC) if it does not contain an infinite strictly increasing sequence. We say $\Phi$ satisfies the \textit{descending chain condition} (DCC) if it does not contain an infinite strictly decreasing sequence.
\end{definition}

\begin{definition}[Contractions]
	We say a projective morphism $f:X\to Z$ between varieties is a \textit{contraction} if $f_*\cO_X=\cO_Z$. In particular, $f$ has connected fibers and if $X\rightarrow Y\rightarrow Z$ is the Stein factorization of $f$, then $Y\rightarrow Z$ is an isomorphism.
\end{definition}

\begin{definition}[Basic Notations]
    Let $X$ be a normal variety, and let $N$ be an $\bR$-divisor on $X$. We often denote the coefficient of a prime divisor $D$ in $N$ by $\mu_DN$. Let $\delta\in \bR$, $\Phi\subset \bR$, and write $N=\sum a_iN_i$ where $N_i$'s are different Weil divisors on $X$. We denote $N\in \Phi$ if $a_i\in \Phi$ for any $i$, and denote $N\geq \delta$ (resp. $D\leq \delta$) if $a_i\geq \delta$ (resp. $a_i\leq \delta)$ for any $i$.
    
	Let $f:X\to Z$ be a contraction between normal varieties. Let $D$ be an $\bR$-divisor on $X$. We say that $D$ is \textit{vertical} over $Z$ if $f(\Supp D)$ is a proper subset of $Z$. We say that $D$ is \textit{horizontal} over $Z$ if the induced map $\Supp D\to Z$ is dominant. Given an $\bR$-divisor $D$ on $X$, there is a unique decomposition $D=D^h+D^v$ such that 
	$D$ is decomposed into its \textit{horizontal part} $D^h$ and the \textit{vertical part} $D^v$ with respect to $f:X\to Z$.
    
    Let $f:X\rightarrow Z$ be a morphism to a normal variety, and let $R$ and $L$ be two $\bR$-Cartier divisors on $X$. We say $R\sim L$ over $Z$ (resp. $M\sim_{\bQ}L$ over $Z$)(resp. $M\sim_{\bR}L$ over $Z$) if there is a Cartier (resp. $\bQ$-Cartier) (resp. $\bR$-Cartier) divisor $N$ on $Z$ such that $R-L\sim f^*N$ (resp. $R-L\sim_{\bQ}f^*N$) (resp. $R-L\sim_{\bR}f^*N$).

    We say an $\bR$-divisor $M$ is \textit{b-Cartier} if it is $\bR$-Cartier and if there is a birational contraction $\phi:X'\rightarrow X$ such that $\phi^*M$ is Cartier.
\end{definition}

\begin{definition}[(Invariant) Iitaka dimension and (invariant) Iitaka volume]
	Let $D$ be an $\bR$-divisor on a projective normal variety $X$. The \textit{Iitaka-Kodaira dimension} $\kappa(D)$ is defined as

    $$\kappa(D)= \limsup_{m\rightarrow \infty}\frac{\mathrm{log}h^0(X,\cO_X(\rounddown{mD}))}{\mathrm{log}m}$$
    if $|mD|\neq \emptyset$ for some $m\in \bZ_{>0}$, and $\kappa(D)=-\infty$ if otherwise, and in case $\kappa =\kappa(D)\geq 0$, the \textit{Iitaka volume}
    $$\mathrm{Ivol}(D):= \limsup_{m\to \infty} \frac{h^0(X,\cO_X(\rounddown{mD}))}{m^\kappa/\kappa!}.$$ If $\kappa(D)=-\infty$, we let $\Ivol(D)=0$.
    
    We now define the \textit{invariant Iitaka dimension} $\kappa_\iota(X,D)$ as follows.  If $|D|_\bR \neq \emptyset$, let $\kappa_\iota(X,D)=\kappa(X,\DD)$ for some $\bR$-divisor $\DD\in |D|_\bR$. Here, the right hand side is the usual Iitaka dimension of $\DD$.  Note that in this case $\kappa_\iota(X,D)$ does not depend on the choice of $\DD$ by \cite[Corollary 2.1.4]{invariantIitakadimension}. If $|D|_\bR=\emptyset$, let $\kappa_\iota(X,D)=-\infty$.

	  Next we define the \textit{invariant Iitaka volume} $\Ivol_{\iota}(D)$ of $D$ as follows. If $\kappa_\iota(D)\geq 0$, let $\DD$ be an element of $|D|_\bR$, and then
	$$\Ivol_{\iota}(D):=\limsup_{m\to \infty} \frac{h^0(\rounddown{m\DD})}{m^{\kappa_\iota(D)}/\kappa_\iota(D)!}.$$ Note that in this case $\Ivol_{\iota}(D)$ does not depend on the choice of $\DD$ by \cite[Corollary 2.1.4]{invariantIitakadimension}.
	If $\kappa_\iota(D)=-\infty$, then let $\Ivol_{\iota}(D)=0$.
\end{definition}

If $f: X\to Z$ is a contraction between two normal varieties and $D\sim_\bR f^*L$ for some big $\bR$-divisor $L$ on $Z$, then $\Ivol_{\iota}(D)=\vol(L)$. 

~\\
  The following example shows that in the case of $\bR$-divisors, usual Iitaka dimension and thus usual Iitaka volume do not behave well under $\bR$-linear equivalence classes.

\begin{example}
    Let $X:= \bP^1$ and $p,q$ two different points on $X$. For any real number $t>0$, consider the divisor $D(t):=tp-tq$. Clearly we have $D(t)\sim_{\bR}0$ and thus $\kappa_{\iota}(D(t))=0$. However, for any integer $m$, $\rounddown{mD(t)} = \rounddown{mtp-mtq} = \rounddown{mtp}-\roundup{mtq}$. So, if $t\notin \bQ$, $\rounddown{mD(t)} =  \rounddown{mt}p-\roundup{mt}q\sim -p$ and $\mathrm{deg}(\rounddown{mD(t)})=-1$ for all $m$, this implies $\kappa(D(t))=-\infty$.
\end{example}

From now on, for simplicity of notation, when we talk about the Iitaka volume of some $\bR$-divisor $D$, we mean the invariant Iitaka volume, and we still denote it $\Ivol(D)$.

\begin{definition}[b-divisors]
	Let $X$ be a normal variety. A \textit{b-divisor} $\textbf{M}$ is a collection of $\bR$-divisors $M_Y$ on $Y$ for each birational contraction $Y\to X$ from a normal variety and satisfies the following: if $\YY\to Y\to X$ are birational contractions, then the pushdown of $M_{\YY}$ on $Y$ is $M_Y$.
	
	We say a b-divisor $\textbf{M}$ is \textit{b-$\bR$-Cartier} if there is a birational contraction $Y\to X$ such that 
	\begin{itemize}
		\item $M_Y$ is $\bR$-Cartier, and
		\item if $\YY\to Y$ is a birational contraction, then $M_{\YY}$ is the pullback of $M_Y$.
	\end{itemize}
	In this case, we say that the b-$\bR$-Cartier divisor $\textbf{M}$ descends on $Y$ and is represented by $M_Y$. Note that the representation is  not unique, if $\YY \to X$ is another birational contraction and $M_{\YY}$ is an $\bR$-Cartier divisor on $\YY$, then $M_Y$ and $M_{\YY}$ define the same b-$\bR$-Cartier b-divisor if the pullbacks of $M_Y$ and $M_{\YY}$ to a common resolution of $Y$ and $\YY$ are the same.

    We say a b-divisor $\textbf{M}$ is \textit{NQC} if it can be written as an $\bR_{\geq 0}$-linear combination of nef b-Cartier b-divisors.
\end{definition}

\subsection{(Generalized) Pairs and Singularities}
\begin{definition}[Pairs and Singularities]
	Let $X$ be a normal quasi-projective variety and $B$ be an $\bR$-divisor on $X$. We say that $(X,B)$ is a \textit{sub-pair} if $K_X+B$ is $\bR$-Cartier. If in addition $B\geq 0$, then $(X,B)$ is a \textit{pair}.
	
	Let $D$ be a prime divisor over $X$, i.e. there is a birational model over $X$ such that $D$ is a prime divisor on this model. Let $W\to X$ be a log resolution of a sub-pair $(X,B)$ so that $D$ is a prime divisor on $W$. Let $K_W+B_W$ be the pullback of $K_X+B$. Define the \textit{log discrepancy} of the prime divisor $D$ as $1-\mu_D B_W$, where $\mu_D B_W$ means the coefficient of $D$ in $B_W$. We denote the log discrepancy of $D$ with respect to $(X,B)$ as $a(D,X,B)$.
	
	We say that a sub-pair $(X,B)$ is \textit{sub-klt} (resp. \textit{sub-lc}, \textit{sub-$\epsilon$-lc}) if $a(D,X,B)>0$ (resp. $a(D,X,B)\geq0$, $a(D,X,B)\geq \epsilon$) for every prime divisor $D$ over $X$. If $(X,B)$ is a pair, then we remove the sub and say the pair is klt (resp. lc, $\epsilon$-lc).
	
	Let $(X,B)$ be a sub-pair. A \textit{non-klt place} (resp. \textit{non-lc place}) is a prime divisor $D$ over $X$ such that $a(D,X,B)\leq 0$ (resp. $a(D,X,B)<0$). A \textit{non-klt center} (resp. \textit{non-lc center}) is the image of a non-klt place (resp. non-lc place).  The \textit{non-klt locus} (resp. \textit{non-lc locus}) of $(X,B)$ is the union of all non-klt places (resp. non-lc places) of $(X,B)$ and denoted as $\Nklt(X,B)$ (resp. $\Nlc(X,B)$).
\end{definition}

\begin{definition}[Generalized pairs and Singularities, {\cite[Definition 1.4, Definition 4.1]{effectiveIitaka}}]
	A \textit{generalized sub-pair} consists of
	\begin{itemize}
		\item a normal variety $X$ equipped with a projective morphism $X\to Z$,
		\item an $\bR$-divisor $B$ on $X$, and
		\item a b-$\bR$-Cartier b-divisor over $X$, represented by a projective birational morphism $\pi: \XX\to X$ and an $\bR$-Cartier $\bR$-divisor $\MM$ on $\XX$
	\end{itemize}
	such that $\MM$ is nef over $Z$ and $K_X+B+M$ is $\bR$-Cartier, where $M:=\pi_*\MM$. If in addition $B\geq 0$, then $(X,B+M)$ is a \textit{generalized pair}. Since a b-$\bR$-Cartier b-divisor is defined birationally, in practice we will often replace $\XX$ with a higher model and replace $\MM$ with its pullback.  In this article, we omit $Z$ but say the generalized pair is projective when $Z$ is a point. We also remark that in this article we often require the b-$\bR$-Cartier b-divisor $\textbf{M}$ to be NQC, so that we can write $\MM =\sum \mu_i M'_i$ with $\mu_i>0$, $M'_i$ are nef Cartier.
 
	Let $D$ be a prime divisor over $X$. Replace $\XX$ with a log resolution of $(X,B)$ such that $D$ is a prime divisor on $\XX$. We can write 
	$$K_{\XX}+\BB+\MM=\pi^*(K_X+B+M).$$ 
	Then we define the \textit{generalized log discrepancy} of  $D$ to be $a(D,X,B+M)=1-\mu_D \BB$.
	
	We say that $(X,B+M)$ is \textit{generalized klt} (resp. \textit{generalized lc}, \textit{generalized $\epsilon$-lc}) if $a(D,X,B+M)>0$ (resp. $a(D,X,B+M)\geq0$, $a(D,X,B+M)\geq \epsilon$) for every prime divisor $D$ over $X$. 

\end{definition}

We recall the following useful result that helps to bound the generalised log canonical threshold uniformly.

\begin{thm} [{\cite[Lemma 2.9]{BVGP}}]\label{bdd glct}
    Let \( d, r \in \mathbb{N} \) and \( \epsilon \in \mathbb{R}^{>0} \). Then there exists \( t \in \mathbb{R}^{>0} \) depending only on \( d, r, \epsilon \) satisfying the following. Assume that

\begin{itemize}
    \item \( (X, B + M) \) is a projective generalised \( \epsilon \)-lc pair of dimension \( d \) with data
    \[ X' \xrightarrow{\phi} X \text{ and } M', \]
    \item \( A \) is a very ample divisor on \( X \) with \( A^d \leq r \),
    \item \( N' \) is a nef \( \mathbb{R} \)-divisor on \( X' \) and \( D \) is an effective \( \mathbb{R} \)-divisor on \( X \),
    \item \( D + N \) is \( \mathbb{R} \)-Cartier where \( N = \phi_* N' \), and
    \item \( A - (B + M + N + D) \) is pseudo-effective.
\end{itemize}

Then
\[
(X, B + tD + M + tN)
\]
is generalised klt with nef part \( M' + tN' \).
\end{thm}

\subsection{(Generalised) canonical bundle formula}
We first recall the construction of usual adjunction formula for fiber spaces based on \cite{kawamataadjunctionformula, ambroadjunctionformula,ambromoduli}.
Let $(X,B)$ be a projective sub-pair and let $f:X\to Z$ be a contraction between quasi-projective normal varieties with $\dim Z>0$ such that $(X,B)$ is sub-lc near the generic fiber of $f$ and $K_X+B\sim_{\bR}0/Z$. 

Fix a prime divisor $D$ on $Z$ and let $t_D$ be the lc threshold of $f^*D$ with respect to $(X,B)$ over the generic point of $D$, i.e. $t_D$ is the largest number so that $(X,B+t_Df^*D)$ is sub-lc over the generic point of $D$. Now let $b_D=1-t_D$ and by basic argument there are finitely many prime divisors $\DD$ on $Z$ such that $b_\DD\neq0$. Hence we can define $B_Z=\sum b_D D$, where the sum runs over all the prime divisors on $Z$. 

Since $K_X+B\sim_{\bR}0/Z$, there is an $\bR$-Cartier $\bR$-divisor $L_Z$ on $Z$ such that $K_X+B\sim_{\bR}f^*L_Z$. Let $M_Z=L_Z-(K_Z+B_Z)$ and we have the following \textit{adjunction formula} $$K_X+B\sim_{\bR} f^*(K_Z+B_Z+M_Z).$$ We call $B_Z$ the \textit{discriminant divisor} and $M_Z$ the \textit{moduli divisor} of $(X,B)$ with respect to $f:X\to Z$. Note that $B_Z$ is uniquely determined but $M_Z$ is determined only up to $\bR$-linear equivalence. 

Take a commutative diagram 
\begin{displaymath}
	\xymatrix{
		\XX \ar[d]_{\ff} \ar[r]^\pi  & X \ar[d]^f         \\
		\ZZ \ar[r]^\mu  & Z   }
\end{displaymath}
such that $\mu$ and $\pi$ are birational contractions. Let $K_{\XX}+\BB$ be the pullback of $K_X+B$ on $\XX$ and similarly we can define a discriminant divisor $B_{\ZZ}$ and $L_{\ZZ}=\mu^*L_Z$ gives a moduli divisor $M_{\ZZ}$ so that 
$$K_{\XX}+\BB\sim_{\bR}\ff^*(K_{\ZZ}+B_{\ZZ}+M_{\ZZ}).$$ It is easy to see that $B_Z$ is the pushdown of $B_{\ZZ}$ and $M_Z$ is the pushdown of $M_{\ZZ}$. Therefore, $B_Z$ and $M_Z$ can be regarded as b-divisors.

It is known that when $(X,B)$ is lc over the generic point of $Z$, $(Z,B_Z+M_Z)$ is a generalized pair. Moreover, if $\textbf{M}$ is NQC, then we can find $\textbf{M}^Z$ in $Z$ to be NQC as well.
~\\

Next, we consider the generalised adjunction for fiber spaces as well.

\begin{definition}[Discrepancy b-divisor]
  Let $(X,B+M)$ be a g-sub pair. We define b-divisors $\textbf{A}(X,B+M)$ and $\textbf{A}^*(X,B+M)$ as follows:
  for any birational morphism $f: Y\rightarrow X$, define $\textbf{A}(X,B+M)_Y:=K_Y+M_Y-f^*(K_X+B+M)$ and $\textbf{A}^*(X,B+M) = \textbf{A}(X,B+M)_Y^{>-1}$.
  We define a sheaf on $X$ as follows: let $\pi: Y\rightarrow X$ be any log resolution of $(X,B+M)$ such that $\textbf{M}$ descends to $Y$, we define $\cO_X(\roundup{\textbf{A}^*(X,B+M)}) := \pi_*\cO_Y(\roundup{\textbf{A}_Y})$. It's easy to see that such definition does not depend on the choice of the log resolution.
\end{definition}

\begin{definition}[Generalised lc fibrations]
  Let $(X,B+M)$ be a g-sub-pair and $f:X\rightarrow Z$ a contraction. If
  \begin{enumerate}
      \item $(X,B+M)$ is glc over the generic point of $Z$,
      \item $\mathrm{rank}f_*\cO_X(\roundup{\textbf{A}^*(X,B+M)}) =1$, and
      \item $K_X+B+M\sim_{\bR} 0/Z$,
  \end{enumerate}
  then we say that $f:(X,B+M)\rightarrow Z$ is a glc-trivial fibration. We say $f:(X,B+M)\rightarrow Z$ is a gklt-trivial fibration if we replace the glc condition by gklt.
\end{definition}

 Given a glc-trivial fibration $f: (X,B+M)\rightarrow Z$, we construct the generalised adjunction formula as follows: for any prime divisor $D$ on $Z$, denote by $\eta_D$ the generic point of $D$. We define
 $$ t_D:= \mathrm{sup}\{a\in \bR| (X,B+af^*D+M) \quad \text{is g-sub-lc over}\quad \eta_D\}.$$

 We then define the \textit{discriminant divisor} $B_Z:=\sum_D(1-t_D)D$, where the sum runs over all the prime divisors on $Z$. As the usual pair case, $B_Z$ is a well-defined $\bR$-divisor as there are only finitely many prime divisors $\DD$ on $Z$ such that $b_D=1-t_D\neq 0$.

 Since $K_X+B+M\sim_{\bR}0/Z$, there is an $\bR$-Cartier $\bR$-divisor $L_Z$ on $Z$ such that $K_X+B+M\sim_{\bR}f^*L_Z$. Let $M_Z:= L_Z-K_Z-B_Z$ and it's called the \textit{moduli divisor} of $f:(X,B+M)\rightarrow Z$. Again $B_Z$ is uniquely determined but $M_Z$ is determined only up to $\bR$-linear equivalence.

 We can regard $B_Z$ and $M_Z$ as b-divisors similar to the usual pair case. When we want to emphasize the b-divisor structure, we use the notation $\textbf{B}^Z$ and $\textbf{M}^Z$ instead.

 We recall the following result regarding the g-lc fibration of $\bR$-divisors.

\begin{lem}{\rm(Adjunction formula of glc fibrations with $\bR$-coefficients, \cite[Theorem 2.23]{JLX22})}\label{adjunction formula for R coefficients}
	Let $(X,B+M)$ be a g-sub-pair and $f:(X,B+M)\rightarrow Z$ be a glc-trivial fibration, such that
    \begin{itemize}
        \item either $B\geq 0$ over the generic point of $Z$, or
        \item $\textbf{M}$ is semi-ample over $Z$.
    \end{itemize}
    Then there exists a g-sub-pair $(Z,B_Z+M_Z)$ such that $K_X+B+M\sim_{\bR}f^*(K_Z+B_Z+M_Z)$. Moreover, if $(X,B+M)$ is glc (resp. gklt), then $(Z, B_Z+M_Z)$ is also glc (resp. gklt). If $\textbf{M}$ is NQC, then we can assume $\textbf{M}^Z$ to be NQC as well.
\end{lem}

\subsection{Minimal model program}
We will use standard results of the minimal model program (cf.\cite{BCHM10}). Assume $(X,B)$ is a pair and  $(X,B)\rightarrow Z$ is a projective morphism. Assume $H$ is an ample$/Z$ $\bR$-divisor such that $K_X+B+H$ is nef$/Z$. If $(X,B)$ is klt, we can run an MMP$/Z$ on $K_X+B$ with scaling of $H$. We know that such MMP terminates when $(X,B)$ is klt and if either $B$ or $K_X+B$ is big.

For a generalised lc pair $(X,B+M)$ with data $\pi:\XX \rightarrow X$ and $\MM$, we can also run MMP on $K_X+B+M$ over $Z$ with scaling of an ample divisor if there exists some boundary $\Delta$ such that $(X,\Delta)$ is klt. Such MMP terminates when $B+M$ or $K_X+B+M$ is big. We refer to \cite[Lemma 4.4]{effectiveIitaka} for more details.

The following is a g-pair version of \cite[Theorem 1.1]{HX13}.

\begin{thm}[{\cite[Theorem 1.3]{LX23}}]\label{g-pair special MMP}
  Let $((X, B, \mathbf{M})/U$ be an NQC glc g-pair and $U^0 \subset U$ a non-empty open subset. Let \(X^0 := X \times_U U^0\), \(B^0 := B \times_U U^0\), and \(\mathbf{M}^0 := \mathbf{M} \times_U U^0\). Assume that

\begin{enumerate}
    \item \((X^0, B^0, \mathbf{M}^0)/U^0\) has a good minimal model, and
    \item any glc center of \((X, B, \mathbf{M})\) intersects \(X^0\).
\end{enumerate}

Then \((X, B, \mathbf{M})/U\) has a good minimal model.
\end{thm}

\subsection{Bounded families}
\begin{definition}[Bounded families of couples and pairs]
	A \textit{couple} consists of a projective normal variety $X$ and a reduced divisor $D$ on $X$. We say that two couples $(X,D)$ and $(\XX,\DD)$ are isomorphic if there is an isomorphism $X\to \XX$ mapping $D$ onto $\DD$. 
	
	Let $\cP$ be a set of couples. Assume that 
	\begin{itemize}
		\item there exist finitely many projective morphisms $V^i\to T^i$ of varieties,
		\item $C^i$ is a reduced divisor on $V^i$, and
		\item  for each $(X,D)\in \cP$ there exists an $i$, a closed point $t\in T^i$ and an isomorphism $\phi: V^i_t\to X$ such that $(V^i_t,C^i_t)$ is a couple and $\phi_*C^i_t\geq D$. 
	\end{itemize}
	Then we say that $\cP$ is \textit{bounded}. This is equivalent to say that there is a positive integer $r$ such that for each $(X,D)\in \cP$, we can find a very ample divisor $A$ on $X$ such that $A^{\dim X}\leq  r$ and $D\cdot A^{\dim X-1}\leq r$ (cf. \cite[Lemma 2.20]{BABI}).
	
	A set of projective lc pairs $(X,B)$ is said to be bounded if the set of $(X,\Supp B)$ forms a bounded family of couples. Note that if $B\in \Phi$ where $0$ is not an accumulation point of $\Phi$, e.g. when $\Phi$ is DCC, this is equivalent to the existence of a positive integer $r$, such that for each pair $(X,B)$, there is very ample divisor $A$ on $X$ such that $A^d\leq r$ and $(K_X+B)\cdot A^{d-1}\leq r$.

    Now suppose $\cE\subset \cG_{\text{glc}}(d,\Phi)$ be a subset of generalised pairs $(X,B+M)$ such that $(X,B+M)$ is projective glc with data $\pi:\XX\rightarrow X$ and $\MM$, such that $B\in \Phi$ and $\MM=\sum \mu_iM'_j$ with $\mu_j\in \Phi$ and $M'_j$ are nef Cartier. Assume $0$ is not an accumulation point of $\Phi$, then we say $\cE$ forms a bounded family, if there is a positive integer $r$ such that for each $(X,B+M)\in \cE$, there is a very ample divisor $A$ on $X$ with $A^d\leq r$ and $(K_X+B+M)\cdot A^{d-1}\leq r$.
    Note that this implies $(X,\Supp B)$ form a bounded family of couples, but we cannot control $\Supp M$ in general, as $M$ is not necessarily effective. In practice we can only bound $\Supp M$ up to $\bR$-linear equivalence.
\end{definition}

We recall several useful facts about boundedness of g-pairs.

\begin{definition}[{\cite[Definition 1.1]{BVGP}}]\label{familyofgpairs}
	Let $d\in \bN$ ,$\Phi\subset \bR^{\geq 0}$, and $v\in \bR^{>0}$.
	\begin{enumerate}[itemsep=13pt]
		\item Let $\cF_{gklt}(d,\Phi)$ be the set of projective generalized pairs $(X,B+M)$ with data $\XX\to X$ and $\MM$ such that 
		\begin{itemize}
			\item $(X, B+M)$ is generalized klt of dimension $d$,
			\item $B\in \Phi$,
			\item $\MM=\sum \mu_i\MMi$ where $\MMi$ is Cartier nef and $\mu_i\in \Phi$ for any $i$, and
			\item $K_X+B+M$ is ample.
		\end{itemize}
		
		\item Let $$\cF_{gklt}(d,\Phi,v)\subseteq \cF_{gklt}(d,\Phi)$$ consist of those $(X,B+M)$ such that $\vol(K_X+B+M)=v$. Similarly, let $$\cF_{gklt}(d,\Phi,\leq v)\subseteq \cF_{gklt}(d,\Phi)$$ consist of those $(X,B+M)$ such that $\vol(K_X+B+M)\leq v$.
		
\end{enumerate}
\end{definition}

\begin{thm}[{\cite[Theorem 1.4]{BVGP}, \cite[Lemma 6.6]{VGP}}]\label{bdd gpair}
  Notation as in Definition \ref{familyofgpairs}. Let $d\in \bN$, $\Phi \subset \bR^{\geq 0}$ a DCC set, and $v\in \bR^{>0}$. Then the set $\cF_{gklt}(d,\Phi,v)$ forms a bounded family. Moreover, when $\Phi$ is a finite set, the subset $\cF_{gklt}(d,\Phi,\leq v,\epsilon)\subset \cF_{gklt}(d,\Phi,\leq v)$ such that $(X, B+M)$ is generalised $\epsilon$-lc, forms a bounded family.
\end{thm}

\begin{proof}
    The first statement is just \cite[Theorem 1.4]{BVGP}, while the second statement is proved only for $\bQ$-divisors in \cite[Lemma 6.6]{VGP}. We show that it holds for $\bR$-divisors as well.
    Pick a generalized pair $(X,B+M)\in  \cF_{gklt}(d, \Phi,\leq v,\epsilon)$ with data $\XX\to X$ and $\MM$. By \cite[Theorem 3.15]{CGD20} there is a rational number $0<\epsilon'<\epsilon$, a finite set $\Psi$, and a positive integer $p$ depending only on $d,\Phi,\epsilon$ that we can write
	$$K_X+B+M=\sum_{i=1}^l r_i(K_X+B_i+M_i)$$ such that
	\begin{itemize}
		\item $\sum_{i=1}^l r_i=1$ and $r_i\in \Psi$,
		\item $(X,B_i+M_i)$ is a generalized $\epsilon'$-lc pair with nef part $\MMi$ on $\XX$ for any $i$, 
		\item $\MM=\sum_{i=1}^l r_i \MMi$, and
		\item $p(K_X+B_i+M_i)$ is integral, and $p\MMi$ is Cartier nef for any $i$.
	\end{itemize}

    Let $I:=\{\frac{i}{p}|i\in \bN\}$. Since $\vol(r_i(K_{X}+B_i+M_i))\leq \vol(K_X+B+M)\leq v$, we know that $\vol(K_{X_i}+B_i+M_i)\leq \frac{v}{r_i^d}=: v_i$. So in particular $(X,B_i+M_i)\in \cF_{gklt}(d,I,\leq v_i, \epsilon')$ and it belongs to a bounded family. In particular, there is a very ample divisor $H$ and a bounded positive integer $r$ such that $H^d\leq r$, $(K_X+B_i+M_i)\cdot H^{d-1}\leq r$ for all $i$. Thus $(K_X+B+M)\cdot H^{d-1}=(\sum_ir_i(K_X+B_i+M_i))\cdot H^{d-1} \leq (\sum_i r_i)\cdot r=r $ as well. This implies $\cF_{gklt}(d, \Phi,\leq v,\epsilon)$ also forms a bounded family.
\end{proof}

\begin{lem}\label{2.17}
Let $d,p,r\in \bN$ and $\Phi \subset \bR^{\geq 0}$ a DCC set. Consider a subset $\cE\subset\cF_{gklt}(d,\Phi)$ such that $pM'$ is Cartier nef, and for every $(X,B+M)\in \cE$, there is a very ample divisor $H$ such that $H^d\leq r$ and $(K_X+B+M)\cdot H^{d-1} \leq r$. Then there is a bounded $l\in \bN$ such that $lH-(K_X+B+M)$ is ample. 
\end{lem}

\begin{proof}
    By \cite[Theorem 1.3]{effectiveIitaka}, there is a bounded natural number $m$ depending only on $d,p,\Phi$ such that $|m(K_X+B+M)|$ defines a birational map. Replacing $m$ by $mp$, we may assume $p|m$. In particular, $h^0(m(K_X+B+M))\neq 0$ and there is an integral effective divisor $0\leq D\sim \rounddown{m(K_X+B+M)}$. We may write
    $$m(K_X+B+M)\sim D+F,$$ where $F:= m(K_X+B+M)-\rounddown{m(K_X+B+M)} \geq 0$ is an $\bR$-divisor. It's easy to see that $\Supp F\subset\Supp B$ and thus $F\leq \Supp B$. On the other hand, $D\cdot H^{d-1}\leq m(K_X+B+M)\cdot H^{d-1}\leq mr$, so by \cite[Lemma 4.6]{moduliofalgebraicvarieties}, there is $l_1\in \bN$ such that $ml_1H-D$ is ample (note that $(X, \Supp D)$ is in a bounded set of couples and since $D$ is integral and $D\cdot H^{d-1}$ is bounded, $D$ takes only finitely many coefficients). On the other hand, as $(X,\Supp B)$ forms a bounded set of couples, there is $l'\in \bN$ such that $ml'H-B_i$ is ample for all $i$, where $\Supp B= \bigcup_iB_i$. The number of irreducible components of $\Supp B$ is bounded by a fixed natural number $n$ depending only on $\Phi$ and $r$, take $l_2=nl'$. Then $ml_2H-F=mnl'H-F$ is ample. Now let $l=l_1+l_2$, we have $mlH-m(K_X+B+M)\sim (ml_1H-D)+(ml_2H-F)$ is ample.
\end{proof}
    
\subsection{Uniform decomposition of canonical bundle formula with real coefficients}

In this subsection, we recall the uniform rational polytope for canonical bundle formulas developed in \cite{effectiveadjunctionformulawithrealcoefficients}. This is one of the key ingredients in the proof of Theorem \ref{MAIN THM1} and Corollary \ref{MAIN COR2}. 

\begin{lem}[{\cite[Theorem 3.3]{effectiveadjunctionformulawithrealcoefficients}}]\label{Decomposition of usual div}
Let $d\in \bN$ and $\Phi\subset \bR^{\geq 0}$ a finite set. Then there exists a finite set $\Psi$ depending only on $d, \Phi$ satisfying the following. Assume that
\begin{itemize}
    \item $(X,B)$ is a projective lc pair of dimension $d$, \item $B\in \Phi$, and
    \item $f:X\rightarrow Z$ is a contraction with $K_X+B\sim_{\bR}0/Z$.
\end{itemize}

 Then we have a decomposition $$K_X+B =\sum_{i=1}^lr_i(K_X+B_i)$$   
 such that
 \begin{itemize}
     \item $r_1,\dots,r_l\in \Psi$ and $r_1,\dots,r_l$ are $\bQ$-linearly independent and $\sum_{i=1}^lr_i=1$,
     \item $(X,B_i)$ is lc and $\Nklt(X,B_i)=\Nklt (X,B)$,
     \item $K_X+B_i\sim_{\bQ}0/Z$,
     \item there is $q\in \bN$ depending only on $d,\Phi$ that $q(K_X+B_i)$ is integral for any $i$, and 
     \item if $\textbf{M}$ and $\textbf{M}_i$ are moduli part of the canonical bundle formula of $(X,B)$ and $(X,B_i)$ respectively, then $\textbf{M}=\sum_ir_i\textbf{M}_i$.
 \end{itemize}
\end{lem}

As pointed out in \cite[Remark 3.5]{effectiveadjunctionformulawithrealcoefficients}, the above result can also be extended to the g-pair version, which we shall state below.

\begin{lem}\label{Decomposition of g-div}
Let $d\in \bN$ and $\Phi\subset \bR^{\geq 0}$ a finite set. Then there exists a finite set $\Psi$ depending only on $d, \Phi$ satisfying the following. Assume that
\begin{itemize}    
\item $(X,B+M)$ is a projective generalised lc pair of dimension $d$ with data $\pi:\XX\rightarrow X$ and $\MM$, 
\item $B\in \Phi$, $\MM=\sum \mu_j M'_j$ with $\mu_j\in \Phi$ and $M'_j$ nef Cartier, and   
\item $f:X\rightarrow Z$ is a contraction with $K_X+B+M\sim_{\bR}0/Z$.
\end{itemize}

 Then we have a decomposition $$K_X+B+M =\sum_{i=1}^lr_i(K_X+B_i+M_i)$$    such that 
 \begin{itemize}    
 \item $r_1,\dots,r_l\in \Psi$ and $r_1,\dots,r_l$ are $\bQ$-linearly independent and $\sum_{i=1}^lr_i=1$,     
 \item $(X,B_i+M_i)$ is generalised lc with data $\pi: \XX \rightarrow X$ and $M_i'$, and $\Nklt(X,B_i+M_i)=\Nklt (X,B+M)$,     
 \item $K_X+B_i+M_i\sim_{\bQ}0/Z$,     
 \item there is $q\in \bN$ depending only on $d,\Phi$ that $q(K_X+B_i+M_i)$ is integral for any $i$, $qM'_i$ is nef Cartier, and     
 \item if $\textbf{M}_Z$ and $\textbf{M}_{i,Z}$ are moduli part of the canonical bundle formula of $(X,B+M)$ and $(X,B_i+M_i)$ respectively, then $\textbf{M}_Z=\sum_ir_i\textbf{M}_{i,Z}$. \end{itemize}
 \end{lem}

\section{DCC of Iitaka volume of usual pairs with real coefficients}
In this section we prove Theorem \ref{MAIN THM1}. The key result here is the following Theorem \ref{effective adjunction formula}, where we extend the effective adjunction formula \cite[Lemma 7.4]{BVGP} to the case of real coefficients. 

\subsection{Effective canonical bundle formula with real coefficients}

\begin{thm}\label{effective adjunction formula}
	Let $d\in \bN$, $u\in \bR^{>0}$ and $\Phi\subset \bR^{\geq0}$ be a DCC set. Then there exists a finite set $\Psi \subset \bR^{\geq0}$ depending only on  $d,u, \Phi$ satisfying the following. Assume that
	\begin{itemize}
		\item $(X,B)$ is  a projective lc pair of dimension $d$ and $B\in \Phi$,
		\item $f: X \to Z$ is a contraction with $K_X+B\sim_{\bR} 0/Z$,
		\item $A\in \Phi$ is an effective $\bR$-divisor on $X$ such that over the generic point $\eta_Z$ of $Z$: $(X,B+tA)$ is lc for some $t>0$ and $A$ is relatively semi-ample, and
		\item $\vol(A|_F)=u$ for the general fiber $F$ of $f$. 
	\end{itemize} 
	
	Then there is an adjunction formula 
	\begin{equation*}
		K_X+B \sim_{\bR} f^*(K_Z+B_Z+M_Z)=f^*(\sum_{i=1}^lr_i(K_Z+B_{i,Z}+M_{i,Z}))
	\end{equation*}
	such that $M_{\ZZ}=\sum_{i=1}^l r_iM_{i,\ZZ}$ on some high resolution $\ZZ\to Z$, where $r_i\in \Psi$ and  $M_{i,\ZZ}$ is Cartier nef for any $i$. 
\end{thm}

First, we show that, at least in the case of finite $\bR$-divisors, Theorem \ref{effective adjunction formula} holds.

\begin{lem}\label{finite effective adjunction}
    Same notation as in Theorem \ref{effective adjunction formula}. Assume further that $B\in \Phi'\subset\Phi$ is a finite set. Then Theorem \ref{effective adjunction formula} holds.
\end{lem}

\begin{proof}
    First of all, by Lemma \ref{Decomposition of usual div}, there is a finite set $\Psi\subset \bR^{\geq0}$ and a positive integer $q$ such that we have a decomposition
    $$K_X+B=\sum_{i=1}^lr_i(K_X+B_i)$$
    satisfying the conditions stated in Lemma \ref{Decomposition of usual div}. In particular, $K_X+B_i\sim_{\bQ}0/Z$ for any $i$ and $q(K_X+B_i)$ is integral.

    Let $F$ be a general fiber of $f$, then $K_F+B_{i,F}=(K_X+B_i)|_F\sim_{\bQ}0$. Hence $(F,B_{i,F}),A_F$ is a polarized lc Calabi-Yau pair. Moreover, there is a real number $t>0$ such that $(F,B_{i,F}+tA_F)$ is lc. Therefore, by \cite[Theorem 6.4]{GOPV}, there is a uniform $\lambda>0$ such that $(F,B_{i,F}+\lambda A_F)$ is lc. Now $K_F+B_{i,F}+\lambda A_F$ is semi-ample and big, there is a birational morphism $g:F\rightarrow \FF$.
    Let $B_{i,F}'$ and $A'_F$ be the corresponding push down of $B_{i,F}$ and $A_F$. Then by negativity lemma, we have $K_F+B_{i,F}=g^*(K_{F'}+B_{i,F}')$. Moreover, we have $K_F+B_{i,F}+A_F=g^*(K_{F'}+B_{i,F}'+A'_F)$ and $K_{F'}+B_{i,F}'+\lambda A'_F$ is ample.
Then as $\vol(K_{F'}+B_{i,F}'+\lambda A'_F)=\vol(K_F+B_{i,F}+\lambda A_F)= \lambda^{\dim F}u$ is fixed, $(F', \Supp(B_{i,F}'+A_F'))$ belongs to a bounded family by \cite[Theorem 1.1]{HMX18}. Note that in \cite[Theorem 1.1]{HMX18}, the authors only consider the case where the pair is a $\bQ$-divisor, but the result also hold for $\bR$-divisors as well by the following argument. First, as $(F', B_{i,F}'+\lambda A_F')\in \cF_{lc}(\dim F,\Phi\cup\lambda\Phi, \lambda^{\dim F}u)$, by \cite[Theorem 1.4]{BVGP}, we have the coefficients of $B_{i,F}'+\lambda A_F'$ actually belong to a fixed finite set. Then we apply the argument in page 224, subsection 6.8.4 of \cite{familiesofgeneraltype} to show actual boundedness.

    Hence by \cite[Lemma 7.2]{BVGP}, possibly replacing $q$ with a bounded multiple, we can assume that $q(K_{F'}+B_{i,F}')\sim 0$. Thus $q(K_{F}+B_{i,F})\sim 0$ as well. This implies that we can find a rational function $\alpha_i$ on $X$ such that $q(K_X+B_i)+\Div(\alpha_i)$ is vertical over $Z_i$. Since $$q(K_X+B_i)+\Div(\alpha_i)\sim_{\bQ}0/Z,$$ we see that $q(K_X+B_i)+\Div(\alpha_i)$ is the pullback of a $\bQ$-Cartier $\bQ$-divisor $qL_Z$ on $Z$ by \cite[Lemma 2.5]{CHL23}. Thus, we have the following adjunction formula
		\begin{equation*}
			q(K_X+B_i)\sim qf^*(K_Z+B_{i,Z}+M_{i,Z})
		\end{equation*}
		where $B_{i,Z}$ is the discriminant divisor and $M_{i,Z}=L_Z-K_Z-B_{i,Z}$ is the moduli divisor.

    We claim that there is a uniform positive integer $p$ depending only on $d,u,\Phi$ such that $pM_{i,Z'}$ is Cartier nef for some high resolution $\ZZ \rightarrow Z$. When $A$ is a $\bZ$-divisor, this is ensured by \cite[Lemma 7.4]{BVGP}. In general, when $A$ is an effective $\bR$-divisor with DCC coefficients, we can apply \cite[Lemma 3.1]{Zhu25} instead.

    Now we have an adjunction formula
		\begin{equation*}
			 \quad	K_X+B\sim_{\bR}f^*(K_Z+B_Z+M_Z),
		\end{equation*}
		where $B_Z=\sum_{i=1}^l r_iB_{i,Z}$ and $\textbf{M}=\sum_{i=1}^l r_i\textbf{M}_i$. 	 Let $\Psi^{\prime}=\{\frac{r}{p}|r\in J\}$, then $$M_{\ZZ}=\sum_{i=1}^l \frac{r_i}{p}(pM_{i,\ZZ})$$ where $\frac{r_i}{p}\in \Psi^{\prime}$ and $pM_{i,\ZZ}$ is Cartier nef for $1\leq i \leq l$. Replace $\Psi$ by $\Psi'$ and we are done.
		
\end{proof}

\begin{proof}[Proof of Theorem \ref{effective adjunction formula}]

We follow the proof of \cite[Theorem 3.3]{Zhu25} with small changes.
	\begin{enumerate} [label=\textsl{Step} \arabic{enumi}., wide=13pt, itemsep=13pt]
		
		\item Let $F$ be a general fiber of $f:X\to Z$. Then $K_F+B_F:=(K_X+B)|_F\sim_{\bR}0$ and $(F,B_F)$ is a lc log Calabi-Yau pair. By \cite[Theorem 1.5]{ACCLCT}, since the coefficients of $B_F$ are in a DCC set $\Phi$, they are in a finite set $\Psi\subset \bR^{\geq0}$ depending only on $d,\Phi$. Hence if we denote $B^h$ to be the horizontal/$Z$ part of $B$, then $B^h\in \Psi$.		
		
		\item Our main idea is to run some MMP to obtain a new pair that preserves the horizontal coefficients of $B$, but also have good control of the vertical coefficients.
		  
		  Take a sufficiently high log resolutions of $(X,B)$ and $Z$ as follows:
		\begin{displaymath}
			\xymatrix{
				\XX \ar[d]_{\ff} \ar[r]^\pi  & X \ar[d]^f         \\
				\ZZ \ar[r]^\mu  & Z   }
		\end{displaymath}
		such that $(\XX,\Sigma)$ is log smooth, where $\Sigma$ is the sum of reduced $\pi$-exceptional divisors and the birational transform of $\Supp B$. Write $K_{\XX}+\BB=\pi^*(K_X+B)$. Let $\tilde{B}^v,\tilde{B}^h$ be the vertical/$\ZZ$ part and horizontal/$\ZZ$ part of the birational transform of $B$. Let $E^v,E^h$ be the vertical/$\ZZ$ part and horizontal/$\ZZ$ part of the reduced $\pi$-exceptional divisors. Then  we take an open subset $\UU$ in $\ZZ$ such that 
		\begin{itemize}
			\item $\mu:\ZZ\to Z$ is an isomorphism on $\UU$,
			\item $L:=\ZZ\backslash \UU$ is a reduced divisor on $\ZZ$, and
			\item  $\ff(\Supp(\tilde{B^v}+E^v))\subseteq L$.
		\end{itemize}
		
		Let $\ff^{-1}L$ be the reduction of the inverse image of $L$ with respect to $\ff:\XX\to \ZZ$ and  add $\ff^{-1}L$ to $\Sigma$. Possibly replacing $(\XX,\Sigma)$ with a higher birational model, we can assume that the condition $(\XX,\Sigma)$ being log smooth is preserved.
		\item Let $\Gamma^{\prime}=\tilde{B^h}+E^h+\ff^{-1}L$. Replacing $\Psi$ with $\Psi\cup \{1\}$, we have $\Gamma^{\prime}\in \Psi$. Run an MMP on $K_{\XX}+\Gamma^{\prime}$ over $\ZZ$ with scaling of some ample divisor. Since over $\ff^{-1}\UU$, $(X,B)$ is a weak lc model of $(\XX,\Gamma^{\prime})$, hence by \cite[Corollary 3.7]{Bir12}, $(\XX,\Gamma^{\prime})$ has a minimal model over $\UU$. Therefore, by \cite[Theorem 1.9]{Bir12}, the MMP terminates over $\ff^{-1}\UU$ and we reach a model $(W,\Gamma_W)$ such that  $K_W+\Gamma_W\sim_{\bR}0/\UU$.
		
		Now we continue to run the MMP on $K_W+\Gamma_W$ over $\ZZ$. The MMP does not modify $W$ over $\UU$. Moreover, the MMP is also an MMP on $K_W+\Gamma_W-aF_W$ where $F_W$ is the pullback of $L$ with respect to $h:W\to \ZZ$ and $a>0$ is a small number. Note that $K_W+\Gamma_W-aF_W$ is semi-ample over $\UU$ and any non-klt center of $(W,\Gamma_W-aF_W)$ intersects with $h^{-1}U'$. The MMP terminates with a good minimal model $V$ by \cite[Theorem 1.2]{Hashizume19}. Let $g: V\to \ZZZ$ be the contraction induced by the semi-ample/$\ZZ$ $\bR$-divisor $K_V+\Gamma_V$ and denote by $\mu^{\prime}$ the morphism $\ZZZ\to \ZZ$. Here it's important to note that $\ZZZ\to \ZZ$ is birational by construction. If we denote $K_V+B_V$ as the pushdown of $K_{\XX}+\BB$, then $\Supp(\Gamma_V-B_V)$ maps into $L\subseteq \ZZ$. Since $K_{\XX}+\BB\sim_{\bR}0/\ZZ$, by negativity Lemma, the pullbacks of $K_V+B_V$ and $K_X+B$ to a common resolution are the same. Therefore, we conclude that $(V,B_V)$ is a sub-klt pair and $K_V+B_V\sim_{\bR}0/Z$. Let $A_V$ be the birational transform of the horizontal/$Z$ part of $A$. Let $G$ be the general fiber of $g:V\to \ZZZ$. Since over $\UU$, $(V,\Gamma_V)$ is a small $\bQ$-factorialization of $(X,B)$, $A_V$ is the pullback of $A$. Therefore, $A_V$ is relatively semi-ample over the generic point of $\ZZZ$ and $\vol(A_V|_G)=u$.  
		
		\item Applying Lemma \ref{finite effective adjunction} to $(V,\Gamma_V)$ over $\ZZZ$, there exists a finite set $\Psi'\subset \bR^{\geq0}$ depending only on $d,u,\Psi, \Phi$ such that we can write an adjunction formula
		\begin{equation*}
			 \quad K_V+\Gamma_V\sim_{\bR} g^*(K_{\ZZZ}+\Gamma_{\ZZZ}+M_{\ZZZ})
		\end{equation*}
		such that $M_{\ZZZZ}=\sum_{i=1}^l r_iM_{i,\ZZZZ}$ on some high resolution $\ZZZZ\to \ZZZ$, where $r_i\in \Psi'$ and  $M_{i,\ZZZZ}$ is Cartier nef for any $i$.
		
		Since $K_V+\Gamma_V$ is the pullback of  $K_X+B$ over $\UU\subseteq \ZZ$, $\Gamma_V-B_V$ is vertical over $\ZZZ$. Since $$K_V+\Gamma_V\sim_{\bR}K_V+B_V\sim_{\bR}0/\ZZZ,$$ we conclude that $\Gamma_V-B_V$ is the pullback of an effective $\bR$-Cartier $\bR$-divisor $P_{\ZZZ}$ on $\ZZZ$ by \cite[Lemma 2.11]{effectiveadjunctionconjecture}. The adjunction formula above induces an adjunction formula
		\begin{equation*}
			 \quad K_V+B_V\sim_{\bR} g^*(K_{\ZZZ}+B_{\ZZZ}+M_{\ZZZ})
		\end{equation*}
		 where $B_{\ZZZ}:=\Gamma_{\ZZZ}-P_{\ZZZ}$ and the moduli part is preserved. 
		
		Since the pullbacks of $K_V+B_V$ and $K_X+B$ to a common resolution are the same, the above  adjunction formula of $(V,B_V)\to Z''$ induces the following adjunction formula 
		\begin{equation*}
			 \quad K_X+B\sim_{\bR} f^*(K_Z+B_Z+M_Z)
		\end{equation*}
		where $K_Z+B_Z+M_Z$ is the pushdown of $K_{\ZZZ}+B_{\ZZZ}+M_{\ZZZ}$. 
        In particular, since $\ZZZZ\to\ZZZ\to\ZZ\to Z$, after replacing $Z'$ by $Z'''$ and replacing $\Psi$ by $\Psi'$ we see $M_{Z'}=\sum_{i=1}^lr_iM_{i,Z'}$ with $r_i\in \Psi$ and $M_{i,Z'}$ is integral thus Cartier nef for any $i$.
	\end{enumerate}
\end{proof}

\subsection{Proof of Theorem \ref{MAIN THM1}}
\begin{proof}[Proof of Theorem \ref{MAIN THM1}]
    First we treat the lc case. Let $(X,B)\in \cI_{lc}(d,\Phi,u)$ and let $f:X\rightarrow Z$ such that $K_X+B\sim_{\bR}0/Z$. Let $F$ be a general fiber of $f$. By assumption, there is an $\bR$ divisor $0\leq A\in \Phi$ on $X$ such that $\vol(A|_F)=u$ is fixed and over some non-open subset of $Z$: $(X,B+tA)$ is lc for some $t>0$ and $A$ is semi-ample.

    Then by Theorem \ref{effective adjunction formula}, there is a finite set $\Psi\subset \bR^{\geq 0}$ depending only on $d,u,\Phi$ such that we can write an adjunction formula
    $$K_X+B\sim_{\bR}f^*(K_Z+B_Z+\sum_{i=1}^lr_iM_{i,Z})$$
    where $\sum_{i=1}^lr_i=1$ and each $r_i\in \Psi$, and for some sufficiently high resolution $\ZZ\to Z$, $M_{Z'}= \sum_{i=1}^lr_iM_{i,Z}'$ and $M_{i,Z'}$ is Cartier nef for any $i$. Moreover, $(Z, B_Z+M_Z)$ is g-lc and 
    $$\Ivol (K_X+B)=\vol (K_Z+B_Z+M_Z)$$
    as $\kappa(K_X+B)=\dim Z$. By the definition of $B_Z$ and the ACC of lct \cite[Theorem 1.1]{ACCLCT}, the coefficients of $B_Z$ belong to a DCC set $\Psi'$ depending only on $d,\Phi$. Replace $\Psi$ by $\Psi\cup\Psi'$, we have $(Z,B_Z+M_Z)\in \cG_{glc}(\dim Z, \Psi)$. Hence by \cite[Theorem 1.3]{BVGP} we have $\Ivol(K_X+B)=\vol(K_Z+B_Z+M_Z)$ belongs to a DCC set depending only on $\dim Z, \Psi$, which in turn depending only on $d,u,\Phi$.

    Now consider the klt case. One may simply use \cite[Theorem 3.3]{Zhu25} instead of Theorem \ref{effective adjunction formula}, but we provide an alternative approach here, at least when $A$ belongs to a fixed finite set of real numbers. Let $(X,B)\in \cI_{klt}(d,\Phi,\leq u)$, let $f:X\to Z$ such that $K_X+B\sim_{\bR}0/Z$. Let $F$ be a general fiber of $f$. Then $(F, B_F)$ is a klt log Calabi-Yau pair. By assumption, $\vol (A|_F)\leq u$ and $A|_F$ is semi-ample.

    Now $(F,B_F)$ is klt and $K_F+B_F\sim_{\bR}0$ and $B_F\in \Phi$ is a DCC set. By \cite[Lemma 2.48]{BABI}, $(F, B_F)$ is $\epsilon$-lc for some fixed $\epsilon$ depending only on $d,\Phi$. By \cite[Theorem 6.2]{GOPV}, $(F,\Supp(B_F+A_F))$ belongs to a bounded set of couple. In particular, any $\bQ$-Cartier Weil divisor on $F$ has bounded Cartier index by \cite[Lemma 2.24]{BABI}.

    There is a fixed $\lambda >0$ such that $(F,B_F+\lambda A_F)$ is klt. By \cite[Theorem 3.15]{CGD20}, there is a uniform decomposition 
    $$K_F+B_F+\lambda A_F=\sum_{i=1}^lr_i(K_F+C_{i,F})$$
    such that $r_i\in I$ where $I$ is a finite set depending only on $d,\Phi,\lambda$ and $\sum_{i=1}^lr_i=1$. We can assume $K_F+C_{i,F}$ is semi-ample for any $i$. Moreover, there is a positive integer depending only on $d,\Phi,\lambda$ such that $p(K_F+C_{i,F})$ is integral $\bQ$-Cartier for any $i$.

    Now possibly replacing $p$ by a bounded multiple, we assume $p(K_F+C_{i,F})$ is Cartier for any $i$, and we have
    $$(\sum_{i=1}^l\frac{r_i}{p}(p(K_F+C_{i,F})))^d=\vol (\lambda A_F)\leq \lambda^{\dim F}u.$$
    By expanding the intersection number on the left side, we see that the expression is a discrete set in $\bR^{>0}$, which implies $\vol (A|_F)$ takes only finitely many possible values. Thus the DCC of Iitaka volume follows from the DCC of the lc case.
\end{proof}

\section{Several facts about generalised pairs}
In this section we prove some boundedness type results about generalised pairs. These results are of independent interest and will be used to prove Theorem \ref{MAIN THM2} and Corollary \ref{MAIN COR1}. 

\subsection{Divisors with generalised log discrepancy close to zero.}

We first state and prove a g-pair version of \cite[Lemma 2.48]{BABI}.

\begin{lem}\label{g-pair version of 2.48}
   Let $d\in \bN$ and $\Phi \subset \bR^{\geq0}$ a DCC set. Then there is $\epsilon>0$ depending only on $d,\Phi$ such that if $(X,B+M)$ is a projective generalised pair with data $\pi:\XX\to X$ and $M'$ and $D$ is a prime divisor over $X$ satisfying
   \begin{itemize}
       \item $(X,B+M)$ is g-lc of dimension $d$ and $(X,0)$ is klt,
       \item $K_X+B+M\sim_{\bR}0$ and $B\in\Phi$ and $M'=\sum\mu_jM'_j$ with $\mu_j\in\Phi$ and $M'_j$ nef Cartier for any $i$, and
       \item $a(D,X,B+M)<\epsilon$,
   \end{itemize}
   then $a(D,X,B+M)=0$.
\end{lem}

\begin{proof}
    If the lemma doesn't hold, then there is a strictly decreasing sequence $\epsilon_i>0$ and a sequence $(X_i,B_i+M_i), D_i$ as in the statement such that $0<a(D,X,B+M)<\epsilon_i$. Now if $D_i$ is already a divisor on $X_i$, let $\phi_i:X_i'\to X_i$ be the identity morphism. Otherwise, since $(X_i,0)$ is klt, we can find a birational morphism $\phi_i:X'_i\to X_i$ extracting only $D_i$. This is ensured by the existence and termination of generalised klt MMP. More precisely, we can first take a log resolution $g:W\to (X,B+M)$ such that $D$ is a divisor on $W$. Take a positive number $\lambda<1$ such that $(X,\lambda B+\lambda M)$ is g-klt. Then we have 
    $$K_W+E_{\lambda}+\lambda M_W=g^*(K_X+\lambda B+\lambda M)+F_{\lambda},$$
    where $E_{\lambda}$ and $F_{\lambda}$ are effective and have no common components. Take $G$ to be the support of the exceptional locus of $g$. Then there is a sufficiently small $t>0$ such that $(W,E_{\lambda}+t(G-D)+\lambda M_W)$ is g-klt. The $(K_W+E_{\lambda}+t(G-D)+\lambda M_W)$-MMP/$X$ terminates to the required model $X'$ by \cite[Lemma 4.4]{effectiveIitaka}.

    Now back to our situation here. Let $K_{X_i'}+B_i'+M_i'$ be the crepant pullback of $K_{X_i}+B_i+M_i$, and let $b_i= 1-a(D_i,X_i,B_i+M_i)$, which is the coefficient of $D_i$ in $B_i'$. Then $B_i'\in \Phi':=\Phi\cup\{b_i|i\in \bN\}$. We can assume $\Phi'$ is a DCC set. Now by the global ACC theorem for g-pairs \cite[Theorem 1.6]{effectiveIitaka} we get a contradiction, as $\Phi'$ is only DCC but not finite.
\end{proof}

\subsection{Uniform lc threshold for g-lc Calabi-Yau generalised pairs}

We also provide a g-pair version of \cite[Theorem 6.4]{GOPV}.

\begin{thm}\label{g-pair version of 6.4}
   Let $d\in \bN$, $v,\delta\in \bR^{>0}$, and $\Phi\subset \bR^{\geq0}$ be a DCC set. Then there is a positive real number $t$ depending only on $d,v,\delta,\Phi$ satisfying the following. Assume that
   \begin{itemize}
       \item $(X,B+M)$ is a g-lc Calabi-Yau pair of dimension $d$ with data $\pi:\XX\to X$ and $M'$,
       \item $B\in\Phi$ and $M'=\sum\mu_jM'_j$ with $\mu_j\in\Phi$ and $M'_j$ nef Cartier for any $i$,
       \item $N\geq 0$ is a nef and big $\bR$-divisor on $X$ with coefficients $\geq \delta$,
       \item $(X,B+uN+M)$ is g-lc for some $u>0$, and
       \item $\vol (N)\leq v$.
   \end{itemize}
   Then $(X,B+tN+M)$ is g-lc.
\end{thm}

\begin{proof}
    \begin{enumerate} [label=\textsl{Step} \arabic{enumi}.,wide=13pt,itemsep=13pt]
    \item We first consider the case where $\Phi\subset \bQ^{\geq 0}$ is a DCC set of rational numbers. In this step we make some basic reductions. Let $(X'',B''+M'')$ be a $\bQ$-factorial g-dlt model of $(X,B+M)$ and let $N''$ be the pullback of $N$. By construction, any exceptional prime divisor over $X$ has coefficient equals to $1$ in $B''$. Now since $(X'',B''+uN''+M'')$ is g-lc for some $u>0$, we see that $\Supp N''$ doesn't contain any exceptional divisor over $X$ so that $N''$ is just the birational transform of $N$, thus $N''\geq \delta$ still holds. Replacing $(X,B+M),N$ with $(X'',B''+M''),N''$ we can assume $(X,0)$ is $\bQ$-factorial klt.

    Now since $K_X+B+M\sim_{\bR}0$, $B\in \Phi$ and $\mu_j\in\Phi$ where $\Phi$ is a DCC set, there is a real number $\epsilon >0$ depending only on $d,\Phi$ such that if $D$ is a prime divisor over $X$ with $a(D,X,B+M)<\epsilon$, then $a(D,X,B+M)=0$, by Lemma \ref{g-pair version of 2.48}. In particular, if $a(D,X,0)<\epsilon$, then $a(D,X,B+M)=0$. Moreover, by global ACC of g-pairs \cite[Theorem 1.6]{effectiveIitaka}, we see the coefficients of $B$ belong to a finite set, and those $\mu_j$ where $M_j'$ is not numerically trivial are finite. Since those $M'_j\equiv 0$ will not contribute to the singularities of $(X,B+M)$, so we may simply ignore such $M_j'$'s and assume there is a positive integer $p$ depending only on $d,\Phi$ such that $pB$ is integral, and $pM'$ is Cartier, as we have assumed $\Phi$ to be a DCC set of rational numbers. Note that after discarding all $M'_j\equiv 0$, we may only have $K_X+B+M\equiv 0$, but this is sufficient for the rest of the proof.

    Let $X'''\to X$ to extract exactly all prime divisors $D$ over $X$ with $a(D,X,0)<\epsilon$. Such extraction exists as $(X,0)$ is klt. By construction we can see $(X''',0)$ is $\epsilon$-lc, and if $N'''$ is the pullback of $N$, it is just the birational transform of $N$. Replacing $(X,B+M),N$ by $(X''',B'''+M'''),N'''$ we may assume $(X,0)$ is $\epsilon$-lc.

    \item In this step we first show birationally boundedness of $(X,\Supp(B+N))$. Since $X$ is $\epsilon$-lc and $N\geq \delta$ is nef and big, and $N-K_X\equiv N+B+M$ is big, by \cite[Theorem 4.2]{GOPV}, there exists $m,l\in \bN$ depending only on $d,\epsilon, \delta$ such that $|mK_X+lN+2pM|$ defines a birational map. Take $L\geq 0$ to be an element in the linear system $|mK_X+lN+2pM|$, and after replacing $l$ by a bounded multiple, we may assume $L\geq N+\Supp N$. Let $R:=L+m\Delta$ where $\Delta$ is a small $\bR$-divisor such that $R$ is a $\bQ$-divisor, and $(X,\Delta)$ is $\frac{\epsilon}{2}$-lc.

    Regard $(X,\Delta+\frac{1}{m}(lN+2pM))$ as a g-pair with nef part $\frac{1}{m}(lN+2pM)$, running an MMP on $K_X+\Delta+\frac{1}{m}(lN+2pM)\sim_{\bQ}\frac{1}{m}R$ ends with a minimal model $Y$, as $R$ is big. Now on $Y$, $R_Y$ is a nef and big $\bQ$-divisor. Moreover, $R_Y-(K_Y+B_Y)\equiv R_Y+M_Y$ is big, as $M_Y$ is at least pseudo-effective. Note that

    $\vol(R_Y)=\vol(R)=\vol(mK_X+m\Delta+lN+2pM)=\vol(-mB-(m-2p)M+m\Delta + lN)\leq \vol ((m+l)N)\leq (m+l)^dv.$

    We claim that $\mu_D(B_Y+R_Y)\geq1$ holds for any component $D$ of $R_Y$. This is obvious when $D$ is not a component of the fractional part of $R_Y$. And for any component $D$ of $R_Y$ that is a component of the fractional part of $R_Y$,
    $$\mu_D(B_Y+R_Y)\geq \mu_D(R_Y)\geq \mu_D(L_Y)\geq\mu_D(\Supp N_Y)=1.$$

    Therefore by applying \cite[Lemma 4.4]{BABI} to $(Y,B_Y),R_Y$, there is a real number $c>0$ and a bounded set of couples $\cP$ depending only on $d,v,\Phi,m,l$ such that there is a projective log smooth couple $(\bar{X},\bar{\Sigma})\in\cP$ and a birational map $\bar{X}\dashrightarrow Y$ such that
    \begin{itemize}
        \item $\Supp\bar{\Sigma}$ contains the exceptional divisors of $\bar{X}\dashrightarrow Y$ and the birational transform of $\Supp(B_Y+R_Y)$, and
        \item $p: X''''\to \bar{X}$ and $q:X'''' \to Y$ is a common resolution and $\bar{R}:=p_{*}q^*R_Y\leq c$.
    \end{itemize}

    Note that $\bar{\Sigma}$ contains the exceptional divisor of the induced map $\bar{X}\dashrightarrow X$ and the birational transform of $\Supp(B+N)$. After replacing $\pi:\XX\to X$ by a further resolution, we may assume $p: X'\to \bar{X}$ and $q:X' \to Y$ is a common resolution. Since $N$ is nef, by negativity we have
    $$\bar{N}:=p_*\pi^*N\leq p_*q^*N_Y\leq p_*q^*R_Y =\bar{R}\leq c.$$
    Hence $\bar{N}$ is supported in $\bar{\Sigma}$ and $\bar{N}\leq c$. Moreover, as $(\bar{X},\bar{\Sigma})$ belongs to a bounded set, there is a very ample divisor $\bar{H}$ on $\bar{X}$ such that $\bar{H}^d$ is bounded from above, say $r$, and $\bar{H}-\bar{\Sigma}$, $\bar{H}-\bar{R}$ and $\bar{H}-\bar{N}$ are ample.

    \item Now we take the nef part into consideration. Let $K_{X'}+B'+M'=\pi^*(K_X+B+M)$, $N'=\pi^*N$, and let $K_{\bar{X}}+\bar{B}+\bar{M}=p_*\pi^*(K_X+B+M)$, where $\bar{M}=p_*M'$. By negativity, we have $K_{X'}+B'+M'=p^*(K_{\bar{X}}+\bar{B}+\bar{M})$. So $(X', B'+M')$ and $(\bar{X},\bar{B}+\bar{M)}$ are crepant models.
    
   Our goal here is to find another suitable model where $M'$ descends. First we define a klt pair on $\bar{X}$ as follows. Let $\bar{\Theta}:=(1-\frac{1}{p})\bar{\Sigma}$  . Clearly $(\bar{X},\bar{\Theta})$ is log smooth klt. Moreover, $(\bar{X}, \bar{\Theta}+\bar{M})$ is a generalised pair of dimension $d$ with data $\pi: \XX\to X$ and $M'$, where $p\bar{\Theta}$ is integral, and $pM'$ is Cartier. On the other hand, $\bar{H}-\bar{\Theta}$ is big as $\bar{\Theta}\leq \bar{\Sigma}$. And since $$R=L+m\Delta\equiv mK_X+m\Delta +lN+2pM$$ and perhaps after replacing $m, l$ by a bounded multiple, we may assume $mK_X+m\Delta+lN+pM$ is also big, thus there is an effective $\bR$-divisor $E\geq 0$ such that $R\equiv E+pM$. Let $E_Y$ be the pushdown of $E$ to $Y$ and we have $R_Y\equiv E_Y+pM_Y$. Let $\bar{E}=p_*q^*E_Y$. Then $$\bar{H}-p_*q^*{pM_Y}\equiv\bar{H}-\bar{R}+\bar{E}$$ is big. Since $\bar{M}\leq p_*q^*{M_Y}$ because $q^*M_Y\geq M'$ by negativity, we see $\bar{H}-p\bar{M}$ is big as well. Therefore, replacing $\bar{H}$ by a bounded multiple, we can assume $\bar{H}-(\bar{\Theta}+\bar{M})$ is big.

   \item Now by \cite[Proposition 3.12]{BVGP}, we deduce that there is a generalised klt pair $(\hat{X},\hat{\Sigma}+\hat{M})$ with data $\rho:\XX\to \hat{X}$ and $M'$ such that
   \begin{itemize}
       \item $h:\hat{X}\dashrightarrow \bar{X}$ is a birational morphism,
       \item $\hat{\Sigma}:=\Supp\hat{\Theta}$ contains the support of the birational transform of $\bar{\Theta}$ and the reduced exceptional divisor of $h$,
       \item $(\hat{X},\hat{\Sigma})$ belongs to a bounded set of couples, and
       \item $M'$ descends to $\hat{X}$, that is, $M'=\rho^*\rho_*M'$.
   \end{itemize}
   
   From construction, it's easy to see that $\hat{\Sigma}$ contains the exceptional divisors of $\hat{X}\dashrightarrow X$ and the birational transform of $\Supp (B+N)$. $(\hat{X},\hat{\Sigma})$ may not be log smooth, but there is a bounded set of couples $\cQ$ such that we have a log resolution $W\to \hat{X}$ with $(W,\Sigma_W)\in \cQ$, where $\Sigma_W$ is the sum of the birational transform of $\hat{\Sigma}$ and the reduced exceptional divisor of $W\to \hat{X}$. Clearly $M'$ also descends to $W$, and we may replace $(\hat{X},\hat{\Sigma})$ by $(W,\Sigma_W)$.

   \item We argue that it suffice to find $t$ on the new model $W$. Let $N_W:=\rho_*N'$ where $\rho: X'\to W$ is a morphism. We claim that the coefficients of $N_W$ is also bounded from above. Indeed, as pointed in the proof of \cite[Proposition 3.7]{BVGP}, we can find a very ample divisor $H_W$ on $W$ such that $H_W^d\leq r'$ is bounded, and $H_W-h^*\bar{H}$ is big. Consider the pullback $h^*N$ of $N$ by $h:W\rightarrow \bar{X}$, as $\bar{H}-\bar{N}$ is ample, we see $$H_W-h^*\bar{N}=H_W-h^*\bar{H} +h^*(\bar{H}-\bar{N})$$ is big, thus the coefficients of $h^*\bar{N}$ is also bounded from above, say $c'$. Note that $h_*N_W=\bar{N}$ and $N_W$ is the pushdown of a nef divisor on $X'$, which is movable, by the general negativity lemma \cite[Lemma 3.3]{Bir12}, we have $N_W\leq h^*\bar{N}\leq c'$, this proves our claim.

   We have $$K_{X'}+B'+M'=\rho^*(K_W+B_W+M_W),$$ where $K_W+B_W+M_W=\rho_*(K_{X'}+B'+M')$. It's easy to see that $(W, B_W+M_W)$ is a log smooth g-lc pair, and the nef part $M_W$ will not contribute to the singularities. Now since $(X,B+uN+M)$ is g-lc for some $u>0$, $(X',B'+uN'+M')$ is sub-glc, so no component of $N'$ has coefficient $1$ in $B'$. Thus no component of $N_W$ has coefficient $1$ in $B_W$. Note that if $\mu_DB_W<1$, then $\mu_DB_W\leq 1-\epsilon$ by Step 1. Now take a fixed $t>0$ such that $c't\leq \epsilon$. Then for any prime divisor $D$ that $\mu_DB_W<1$, $$\mu_D(B_W+tN_W)\leq 1-\epsilon+tc'\leq 1.$$ Moreover, since $\Supp{\Sigma_W}$ contains $\Supp B_W\cup\Supp N_W$, we see $(W,\Supp B_W\cup \Supp N_W) $ is log smooth. Thus $(W, B_W+tN_W+M_W)$ is sub-g-lc.

   Now on $X'$ we have $$\pi^*(K_X+B+tN+M)\leq \rho ^*(K_W+B_W+tN_W +M_W)$$ by negativity, we deduce that $(X, B+tN+M)$ is lc. Note that $t$ depends only on $d, v, \Phi$ indeed.

   \item We now give a short explanation about the general $\bR$-divisor case. As in Step 1, we can reduce to the case where the coefficients of $B$ and the $\mu_j$'s belong to a fixed finite set $\Lambda$, then by Lemma \ref{Decomposition of g-div}, there is a finite set $\Psi$ depending only on $d,\Lambda$ such that we have a uniform decomposition 
   $$ K_X+B+M=\sum_{i=1}^lr_i(K_X+B_i+M_i),
   $$
   such that $\sum_{i=1}^lr_i=1$ and $K_X+B_i+M_i\sim_{\bQ}0$, and there is a bounded $q\in \bN$ such that $q(K_X+B_i+M_i)$ is integral and $qM_i'$ is Cartier for any $i$. Moreover, $\Nklt(X, B_i+M_i)=\Nklt (X,B+M)$. Thus we can apply the above result for $\bQ$-divisors to find a bounded $t>0$ such that $(X,B_i+tN+M_i)$ is g-lc for any $i$. Then as
   $$ K_X+B+tN+M= \sum_{i=1}^lr_i(K_X+B_i+tN+M_i),$$
   we conclude that $(X,B+tN+M)$ is g-lc as well.
    \end{enumerate}
\end{proof}

\begin{remark}\label{remark about gklt}
    We provide another approach when $(X, B+M)$ is already g-klt, which is much easier. First, there is an $\epsilon>0$ such that $(X,B+M)$ is g-$\epsilon$-lc. Step 1 and Step 2 are the same as in the proof of Theorem \ref{g-pair version of 6.4}. We then find a uniform $t>0$ on $\bar{X}$ such that $(\bar{X},\bar{B}+t\bar{N}+\bar{M})$ is sub-g-klt, without constructing another model where $M'$ descends. Indeed, $(\bar{X},\bar{\Sigma})$ is a bounded set of couples, and there is a very ample divisor $\bar{H}$ with bounded $\bar{H}^d\leq r$, and after replacing $\bar{H}$ by a bounded multiple we may assume $\bar{H}-\bar{B}-\bar{N}-\bar{M}$ is big. Then we may apply Theorem \ref{bdd glct} to find such a $t$. Moreover, we remark that in this case we can take a uniform $t>0$ such that $(X,B+tN+M)$ is g-$\frac{\epsilon}{2}$-lc.
\end{remark}

\section{DCC of Iitaka volume of generalised pairs}
In this section we prove Theorem \ref{MAIN THM2}, Corollary \ref{MAIN COR1} Corollary \ref{MAIN COR2} and Theorem \ref{MAINTHMnew}. As in the usual pair case, the key step is to establish an effective adjunction formula for g-lc trivial fibrations.

\subsection{Effective generalised canonical bundle formula}

\begin{thm}\label{g-pair effective canonical bundle formula}
    Let $d,q \in \mathbb{N}$, $\Phi \subset \mathbb{Q}^{\geq 0}$ a DCC set, and $u \in \mathbb{Q}^{\geq 0}$. Then there exists $p\in \bN$ depending only on $d,q,\Phi,u$ satisfying the following. Assume that
    \begin{itemize}    
    \item $(X, B+M)$ is projective generalised lc of dimension $d$ with data $\phi:X'\rightarrow X$ and $M'$, \item the coefficients of $B$ are in $\Phi$,    
    \item $qM'$ is nef Cartier,    
    \item $f: X \to Z$ is a contraction with $K_X + B +M \sim_{\mathbb{Q}} 0/Z$,         
    \item we have a generalised adjunction formula    $$q(K_X+B+M)\sim qf^*(K_Z+B_Z+M_Z) ,$$    
    \item there is an effective $\bQ$-divisor $A \geq 0$ on $X$ such that $A\in \Phi$ and over some non-empty open subset of $Z$: $(X, B + tA +M)$ is generalised lc for some $t > 0$ and $A$ is semi-ample, and   
    \item $0<\operatorname{vol}(A|_F) \leq u$ for the general fibres $F$ of $f$.
    \end{itemize}
    Then $pM_{Z'}$ is Cartier nef on some high resolution $\ZZ\to Z$.
\end{thm}

We also provide another slightly different version below, where we allow $M'=\sum \mu_j M'_j$ has DCC coefficients $\mu_j$.

\begin{prop}\label{another versiom}
Let $d,q \in \mathbb{N}$, $\Phi \subset \mathbb{Q}^{\geq 0}$ a DCC set, and $u \in \mathbb{Q}^{\geq 0}$. Then there exists $p\in \bN$ depending only on $d,q,\Phi,u$ satisfying the following. Assume that    
\begin{itemize}        
\item $(X, B+M)$ is projective generalised lc of dimension $d$ with data $\phi:X'\rightarrow X$ and $M'$, 
\item the coefficients of $B$ are in $\Phi$,        
\item $M'=\sum \mu_j M_j'$ where $M_j'$ are nef Cartier and $\mu_j\in \Phi$,      
\item $f: X \to Z$ is a contraction with $K_X + B +M \sim_{\mathbb{Q}} 0/Z$, 
\item none of the $M'_j$ is numerically trivial on the general fiber $F$ of $f$,
\item we have a generalised adjunction formula    
$$q(K_X+B+M)\sim qf^*(K_Z+B_Z+M_Z) ,$$        
\item there is an effective $\bQ$-divisor $A \geq 0$ on $X$ such that $A\in \Phi$ and over some non-empty open subset of $Z$: $(X, B + tA +M)$ is generalised lc for some $t > 0$ and $A$ is semi-ample, and       
\item $0<\operatorname{vol}(A|_F) \leq u$ for the general fibres $F$ of $f$.    
\end{itemize}    

Then $pM_{Z'}$ is Cartier nef on some high resolution $\ZZ\to Z$.    
\end{prop}

We first prove the case where $Z$ is a curve.

\begin{lem}\label{curve case}
   Let $d,q \in \mathbb{N}$, $\Phi\subset \bQ^{\geq 0}$ a DCC set, $u \in \mathbb{Q}^{>0}$. Then there exists $p\in \bN$ depending only on $d,q,\Phi,u$ satisfying the following. Assume that    
   \begin{itemize}        
   \item $(X, B+M)$ is projective generalised lc of dimension $d$ with data $\phi:X'\rightarrow X$ and $M'$,  \item $qM'$ is nef Cartier,        
   \item $f: X \to Z$ is a contraction onto a curve $Z$ with $K_X + B +M \sim_{\mathbb{Q}} 0/Z$,             
   \item we have a generalised adjunction formula    
   $$q(K_X+B+M)\sim qf^*(K_Z+B_Z+M_Z) ,$$        
   \item there is an effective $\bQ$-divisor $A\in \Phi$ on $X$ such that over some non-empty open subset of $Z$: $(X, B + tA +M)$ is generalised lc for some $t > 0$ and $A$ is semi-ample, and       
   \item $0<\operatorname{vol}(A|_F) \leq u$ for the general fibres $F$ of $f$.    
   \end{itemize}    
   Then $pM_{Z}$ is integral.
\end{lem}

\begin{proof}

\begin{enumerate} [label=\textsl{Step} \arabic{enumi}.,wide=13pt,itemsep=13pt]
  \item We follow the proof of \cite[Lemma 7.3]{BVGP}. Fix a closed point $z_0\in Z$, it suffice to prove that $\mu_zpM_Z$ is integral for any $z\notin U$, for some $p$ depending only on $d,q,c,u$. Note that as $q(K_F+B_F+M_F)\sim 0$ on the general fiber $F$ of $f$, after shrinking $U$ again we can assume that over $U$, $qB$ is integral without vertical components.

  Let $W\to X$ be a log resolution of $(X,\Supp(B+A+f^{-1}(Z\setminus U)))$ where $M'$ descends to $W$. Let $B_W$ be the sum of the reduced exceptional divisor of $W\to X$ plus the birational transform of the horizontal/$Z$ part of $B$ plus the birational transform of the reduction of the fibres of $f$ over $Z\setminus U$. Let $A_W$ be the birational transform of the horizontal/$Z$ part of $A$. It's easy to see that any component of $A$ is not a component of $\rounddown{B_W}$, thus $(W,B_W+tA_W+M_W)$ is g-lc for any sufficiently small $t>0$.

  \item First, we run an MMP on $K_W+B_W+tA_W+M_W$ over $X$ with the scaling of an ample divisor for some small $t>0$. By construction, over $f^{-1} U$, $B_W+tA_W$ is the reduced sum of the exceptional divisor of $W\to X$ and the birational transform of $B+tA$. Then as $(X,B+tA)$ is lc at least on $f^{-1} U$, the MMP terminates over $U$, and we reach a model $Y$ such that $(Y,B_Y+tA_Y+M_Y)$ is a $\bQ$-factorial g-dlt model of $(X,B+tA+M)$ over $U$. We may assume $A$ does not contain any generalised non-klt center of $(X,B+tA+M)$ over $U$, after decreasing $t$ if needed. Then over $U$, $A_Y$ coincides with the pullback of $A$ and $(Y,B_Y)+M_Y$ is a $\bQ$-factorial g-dlt model of $(X,B+M)$. In particular, over $U$, $K_Y+B_Y+M_Y\sim_{\bQ}0$ and $K_Y+B_Y+tA_Y+M_Y$ is semi-ample.

  Next, we continue running an MMP on $K_Y+B_Y+tA_Y+M_Y$ over $Z$ with the scaling of an ample divisor. The MMP is trivial over $U$. Moreover, the MMP is also an MMP on $K_Y+B_Y+tA_Y-aF_Y+M_Y$ where $F_Y$ is the sum of the fibres of $g:Y\to Z$ over the points in $Z\setminus U$, and $a>0$ is sufficiently small. We claim that this MMP terminates to a good minimal model, say $V$. Indeed, $K_Y+B_Y+tA_Y-aF_Y+M_Y$ is semi-ample over $U$ by the previous paragraph, and any generalised non-klt center of $(Y, B_Y+tA_Y-aF_Y+M_Y)$ intersects $g^{-1}U$, since $g^{-1}(Z\setminus U)=\Supp F_Y$ and if some generalised non-klt center of $(Y, B_Y+tA_Y-aF_Y+M_Y)$ is contained in $\Supp F_Y$, there is a generalised non-klt center of $(W,B_W+tA_W-aF_W+M_W)$ contained in $\Supp F_W$, which is impossible. Now by \cite[Theorem 1.3]{LX23} which is a g-pair version of \cite[Theorem 1.1]{HX13}, the MMP terminates to a good minimal model.

  \item Decreasing $t$ if necessary and running the MMP as the previous step, we may assume that $K_V+B_V+sA_V+M_V$-MMP with scaling of an ample divisor does not contract any divisor, for any $0<s<t$. Thus $K_V+B_V+M_V$ is a limit of movable/$Z$ divisors. In particular, $K_V+B_V+M_V$ is pseudo-effective over $Z$. We claim that $K_V+B_V+M_V\sim_\bQ 0/Z$. Since $K_V+B_V+M_V\sim_\bQ 0$ over $U$, we have $K_V+B_V+M_V\sim_\bQ P_Z/Z$, where $P_Z$ is a vertical over $Z$ divisor. After adding some negative $\bQ$-linear combination of the fibres $V\to Z$, we may assume $P_V\leq 0$ and that $\Supp P_V$ doesn't contain the support of any fibres of $V\to Z$. Now $P_V|_D$ is pseudo-effective for any component $D$ of a fibre of $V\to Z$, but $P_V\leq 0$. This implies $P_V=0$ otherwise we may take $D$ to be a component of a fibre that intersects $P_V$ but not contained in $P_V$, which yields a contradiction. Thus $K_V+B_V+M_V\sim_\bQ 0/Z$ and $A_V$ is semi-ample and big over $Z$.

  Now let $h:V\to Z$ be the corresponding contraction. Let $p:N\to X$ and $p':N\to V$ be a common resolution. By construction we see $$L:=p'^*(K_V+B_V+M_V)-p^*(K_X+B+M)$$
  is zero over $U$, and $L\sim_\bQ 0/Z$ is the pullback of a $\bQ$-divisor $P_Z$ supported in $Z\setminus U$. Consider the generalised adjunction formula for $h:(V, B_V+M_V)\to Z$
  $$q(K_V+B_V+M_V)=q(p'_*(p^*(K_X+B+M)+L))\sim qh^*(K_Z+B_Z+P_Z+M_Z).$$
  By construction, it's easy to see the discriminant part of $(V,B_Z+M_V)$ is precisely $B_Z+P_Z$, and the moduli part is $M_Z$, which is preserved. Therefore, we may replace $(X, B+M), A$ with $(V,B_V+M_V), A_V$ so that we can assume $(X, B+tA+M)$ is $\bQ$-factorial g-dlt for some $t$, $A$ is semi-ample and big over $Z$, $qB$ is integral and $\rounddown{B}$ contains $f^{-1}(Z\setminus U)$. In particular, $\mu_z B_Z=1$ for any $z\in Z\setminus U$.

  \item Now let $X\to T/Z$ be the contraction defined by $A$ over $Z$. Let $S$ be a vertical/$Z$ component of $\rounddown{B}$ that is not contracted over $T$. Then $A_S:=A|_S$ is a well-defined nef and big $\bQ$-divisor and $\vol(A_S)\leq \vol(A|_F)\leq u$ where $F$ is a general fiber of $f$. Consider the adjunction on $S$:
  $$K_S+B_S+M_S:=(K_X+B+M)|_S\sim_{\bQ}0,$$
  we have $(S,B_S+M_S)$ is a generalised log Calabi-Yau pair, with coefficients of $B_S$ belong to a fixed DCC set by \cite[Proposition 4.9]{effectiveIitaka}, which is actually a finite set by \cite[Theorem 1.6]{effectiveIitaka}. On the other hand, $qM_{S'}'$ is still nef Cartier, where $S'\to S$ is a high resolution.

  Note that $(S,B_S+tA_S+M_S)$ is g-lc for some $t>0$. Therefore, by Theorem \ref{g-pair version of 6.4}, there is a uniform rational number $\lambda>0$ depending only on $d,q,\Phi,u$ such that $(S, B_S+\lambda A_S +M_S)$ is g-lc. Moreover, the coefficients of $B_S+\lambda A_S$ belong to a fixed DCC set, so by \cite[Theorem 1.3]{BVGP}, the volumes $\vol (K_S+B_S+\lambda A_S+M_S)$ belong to a DCC set, thus is bounded from below away from $0$, say $\vol (K_S+B_S+\lambda A_S+M_S) \geq c$ for some $c\in \bQ^{>0}$ depending only on $d,q,\Phi,u$.

  \item We now finish the proof. For any $z\in Z\setminus U$, suppose $f^*z=\sum m_iF_i$, where $F_i$ are the irreducible components. Note that
  $$\sum m_i\vol((K_X+B+\lambda A+M)|_{F_i})=\vol((K_X+B+\lambda A+M)|_F)\leq \lambda^{d-1}u$$
  where $F$ is a general fiber of $f$.

  Clearly there is at least one component $F_i$ is not contracted over $T$, say $F_1$. Then as $\vol((K_X+B+\lambda A+M)|_{F_1})\geq c$, we see $cm_1\leq \lambda^{d-1}u$ and thus $m_1\leq \frac{\lambda^{d-1}u}{c}$ is bounded from above.

  Note that as $q(K_X+B+M)\sim qf^*(K_Z+B_Z+M_Z)$ and $\mu_zB_Z=1$, we see $qf^*M_Z$ is integral over $z$. Thus $q(\mu_z M_Z)(\sum m_iF_i)$ is integral, so $qm_1(\mu_z M_Z)$ is integral. As $m_1\leq \frac{\lambda^{d-1}u}{c}$, we may take $p:=q\roundup{\frac{\lambda^{d-1}u}{c}}!$ such that $p\mu_zM_Z$ is integral. Note that $p$ clearly depends only on $q,d,\lambda,u,c$, hence depending only on $d,q,\Phi,u$ in turn.
\end{enumerate}
\end{proof}

\begin{proof}[Proof of Theorem \ref{g-pair effective canonical bundle formula}]
First, by the non-effective version canonical bundle formula for g-pairs Lemma \ref{adjunction formula for R coefficients}, there is a high resolution $\bar{Z} \to Z$ such that $M_{\bar{Z}}$ is nef. Let $\phi: \bar{X}\to X$ be a log resolution such that $\bar{X}\dashrightarrow \bar{Z}$ is a morphism and $M'$ descends to $\bar{X}$. Let $\bar{\Delta}$ be the horizontal/$Z$ part of the reduced exceptional divisor of $\phi$ plus the birational transform of the horizontal/$Z$ part of $B$. So every generalised non-klt center of $(\bar{X},\bar{\Delta}+\bar{M})$ is horizontal over $\bar{Z}$. 

Run an MMP on $K_{\bar{X}}+\bar{\Delta}+\bar{M}$ over $X$ with the scaling of an ample divisor. we claim that such an MMP terminates over the generic point of $\bar{Z}$. As over the generic point, we may write
$$K_{\bar{X}}+\bar{\Delta}+\bar{M}=\phi^*(K_X+B+M) +\bar{E}$$
where $\bar{E}\geq 0$ is an exceptional $\bQ$-divisor on $\bar{X}$. Note that $\bar{X}$ is of Fano type over $X$ as $\bar{X}\to X$ is birational. Since $M'$ descends to $\bar{X}$, we have $\bar{M}$ is big and nef over $X$, thus it's semi-ample/$X$. In particular, there exists an effective $\bQ$-divisor $\bar{C}$ such that $K_{\bar{X}}+\bar{\Delta}+\bar{M}\sim_{\bQ} K_{\bar{X}}+\bar{C}/X$ where $(\bar{X},\bar{C})$ is a dlt pair. Then as 
$$K_{\bar{X}}+\bar{C}\sim_{\bQ} \phi^*(K_X+B+M)+\bar{E}$$
over the generic point, 
and the $(K_{\bar{X}}+\bar{\Delta}+\bar{M})$-MMP/$X$ is also an $(K_{\bar{X}}+\bar{C})$-MMP/$X$, we see the MMP indeed terminates over the generic point of $\bar{Z}$, by \cite[Theorem 1.8]{Bir12}.

So we reach a model $X''$ where $K_{X''}+\Delta''+M''\sim_{\bQ}0$ over the generic point of $\bar{Z}$. So in particular $\Supp \bar{E}$ is contracted over the MMP process. Then by \cite[Theorem 1.3]{LX23}, we may run a further MMP on $K_{X''}+\Delta''+M''$ over $\bar{Z}$ that terminates to a good minimal model, as the generalised non-klt center of $(X'',\Delta''+M'')$ is horizontal over $\bar{Z}$. Replacing $X''$ by the minimal model we may assume $K_{X''}+\Delta''+M''$ is semi-ample over $\bar{Z}$ defining a contraction $X''\to \ZZZ/\bar{Z}$. Note that $\ZZZ \to \bar{Z}$ is birational as $K_{\bar{X}}+\bar{\Delta}+\bar{M}$ has relative Kodaira dimension $0$ over $\bar{Z}$. Moreover, the moduli divisor of $(X'',\Delta''+M'')\to \ZZZ$ coincides with the moduli divisor of $(X,B+M)\to Z$, as by our construction, the pullbacks of $K_X+B+M$ and $K_{X''}+\Delta''+M''$ are equal over the generic point of $Z''$.
It's easy to see that there is an induced adjunction formula
$$q(K_{X''}+\Delta''+M'')\sim q f''^*(K_{Z''}+B_{Z''}+M_{Z''}) $$
where $f'':X''\to \ZZZ$ is the corresponding contraction.

As $\ZZZ$ is a higher model, it suffice to find a bounded $p\in\bN$ such that $pM_{Z''}$ is integral, then $pM_{\bar{Z}}$ will be Cartier automatically.
Cutting $Z''$ by $(\dim Z-1)$-general hyperplane sections, we can find a curve $C\subset Z''$, an g-lc pair $(S,\Gamma+M_S)$ over $C$ with $S\subset X''$ such that $K_S+\Gamma+M_S\sim_\bQ 0/C$, and the coefficients of $\Gamma$ belong to $\Phi$. Moreover, let $A''$ be the birational transform of the horizontal/$Z$ part of $A$ and $H:=A''|_S$. Then over the generic point of $C$, $(S,\Gamma +tH)$ is lc for some $t>0$, and $H$ is big and semi-ample with $\vol (H|_G)\leq u$ for the general fibre $G$ of $g:S\to C$. In addition, we have an induced adjunction formula 

$$q(K_S+B_S+M_S)\sim qg^*(K_C+B_C+M_C)$$
where $M_C:=M_{Z''}|_C$.
By Lemma \ref{curve case}, there is $p\in \bN$ depending only on $d,q,\Phi,u$ such that $pM_C$ is integral. This implies $pM_{Z''}$ is integral as $C\subset Z''$ is general.
\end{proof}

\begin{proof}[Proof of Proposition \ref{another versiom}]
    Note that over the general fibre $F$ of $f:X\to Z$, we have $K_F+B_F+M_F\sim_{\bQ}0$, where $M'_F=\sum \mu_jM_j'|_F$. By \cite[Theorem 1.6]{effectiveIitaka}, we have those $\mu_j$ where $M_j'|_F$ is not numerically trivial belong to a fixed finite set, thus after replacing $q$ by a bounded multiple, we may assume $qM'$ is Cartier. We are done by Theorem \ref{g-pair effective canonical bundle formula}.
\end{proof}

\subsection{Proof of Theorem \ref{MAIN THM2}}

\begin{proof}[Proof of Theorem \ref{MAIN THM2}]
  Pick $(X, B+M)\in \cI_{glc}(d,\Phi,q,<u)$ with data $\pi:\XX\to X$ and nef part $M'$. Let $f:X\to Z$ be the g-lc trivial fibration such that $K_X+B+M\sim_\bQ 0/Z$. By Theorem \ref{g-pair effective canonical bundle formula}, we have a bounded $r\in \bN$ depending only on $d,\Phi,q,u$ such that
  $$q(K_X+B+M)\sim qf^*(K_Z+B_Z+M_Z)$$
  with $pM_{Z'}$ is Cartier nef for any high resolution $\ZZ \to Z$. Moreover, the coefficients of $B_Z$ belong to a fixed DCC set $\Psi$ depending only on $d,\Phi$ by \cite[Theorem 1.5]{effectiveIitaka}. We may assume $\frac{1}{p}\in \Psi$, then $(Z,B_Z+M_Z)\in \cF_{glc}(\dim Z,\Psi)$. We have
  $$\Ivol(K_X+B+M)=\vol(K_Z+B_Z+M_Z)$$
  belongs to a fixed DCC set depending only on $d,\Phi,q,u$ by \cite[Theorem 1.3]{BVGP}.
\end{proof}

\subsection{The Fano type case}
In this subsection we study the Fano type case, where the situation is much easier, as the existence of $q$ and $A$ are automatic.

We also provide an effective adjunction formula in the Fano type setting, where we state the following result for $\bR$-divisors. Here we remark that there is an $\bR$-divisor version of Theorem \ref{g-pair effective canonical bundle formula} as well, but we have omitted it for the simplicity of assumptions.

\begin{lem}\label{g-pair Fanotype eff adjunction}
  Let $d \in \mathbb{N}$, $\Phi \subset \mathbb{R}^{\geq 0}$ a DCC set, $\Phi' \subset \mathbb{R}^{\geq 0}$ a finite set. Then there exists a finite set $\Psi\subset\bR^{\geq 0}$ depending only on $d,\Phi,\Phi'$ satisfying the following. Assume that    
  \begin{itemize}        
  \item $(X, B+M)$ is projective generalised klt of dimension $d$ with data $\phi:X'\rightarrow X$ and $M'$, 
  \item the coefficients of $B$ are in $\Phi$,        
  \item $M'=\sum \mu_j M_j'$ where $M_j'$ are nef Cartier and $\mu_j\in \Phi'$,      
  \item $f: X \to Z$ is a contraction with $K_X + B +M \sim_{\mathbb{R}} 0/Z$, and
  \item $X$ is of Fano type over $Z$, i.e. $-K_X$ is big over $Z$.
  \end{itemize}    
Then there is an effective adjunction formula
  $$K_X+B+M\sim_\bR f^*(K_Z+B_Z+M_Z) = f^*(\sum_{i=1}^lr_i(K_Z+B_{i,Z}+M_{i,Z}))$$   
  such that $M_{Z'}=\sum_{i=1}^l r_iM_{i,Z'}$ on some high resolution $\ZZ \to Z$ where $r_i\in \Psi$ and $M_{i,\ZZ}$ is Cartier nef for any $i$.
\end{lem}

\begin{proof}
   \begin{enumerate} [label=\textsl{Step} \arabic{enumi}.,wide=13pt,itemsep=13pt]
   \item We first show the case when $\Phi\subset \bQ^{\geq 0}$ and $qM'$ is Cartier for some fixed $q\in \bN$. Replacing $X$ by a small $\bQ$-factorialisation and replace $K_X+B+M$ by its crepant pullback, we can assume $X$ is $\bQ$-factorial. We may run an MMP on $-K_X$ over $Z$ so that we obtain a model $X''$ such that $-K_{X''}$ is nef and big over $Z$. Moreover, $(X'',B''+M'')$ is still generalised klt as $K_X+B+M\sim_{\bR}0/Z$ is MMP-trivial. Replacing $(X,B+M)\to Z$ by $(X'',B''+M'')\to Z$, we may assume $-K_X$ is nef and big over $Z$.

   Let $F$ be a general fiber of $f:X\to Z$. Then since $K_F+B_F+M_F\sim_{\bR} 0$ and $B_F\in \Phi$, $\mu_j\in \Phi$, by Lemma \ref{g-pair version of 2.48}, we see there is $\epsilon >0$ depending only on $d, \Phi$ such that $(F,B_F+M_F)$ is g-$\epsilon$-lc. In particular, $(X,0)$ is $\epsilon$-lc and $-K_F$ is nef and big. By \cite{BABII}, we see $F$ belong to a fixed bounded family depending only on $d,\epsilon$, and there is a bounded number $v>0$ such that $\vol(-K_F)\leq v$. Moreover, by \cite[Theorem 1.6]{effectiveIitaka}, the coefficients of $B_F$ belong to a fixed finite set, thus in particular there is a bounded $l$ such that $l(K_F+B_F+M_F)$ is integral. By \cite[Lemma 2.24]{BABI}, after replacing $l$ by a bounded multiple, we may assume $l(K_F+B_F+M_F)$ and $lK_F$ are Cartier divisors. Since $F$ is of Fano type, any Cartier divisor $D\sim_{\bQ} 0$ is actually linearly equivalent to $0$, as $\mathrm{Pic}(F)$ is torsion free by \cite[Proposition 2.12]{fano type}. Thus $l(K_F+B_F+M_F)\sim 0$ and we deduce that $l(K_X+B+M)\sim lf^*L_Z$ for some $\bQ$-Cartier $\bQ$-divisor $L_Z$. In particular, by taking $M_Z:=L_Z-K_Z-B_Z$, we obtain

   $$l(K_X+B+M)\sim lf^*(K_Z+B_Z+M_Z).$$
   In addition, on $\XX$, $lM'$ is nef Cartier. Since $X'\to Z$ is also of Fano type, by effective base-point-free theorem \cite{effectivebasepointfree}, after replacing $l$ by a bounded multiple, we may assume $lM'$ is base-point-free, and there is some effective integral divisor $0\leq D'\sim lM'$. Let $A:=B+\frac{1}{l}D$, where $D$ denotes the pushdown of $D'$ to $X$. Then $lA$ is effective integral, and $lA\sim lB+lM\sim -lK_X$, and $0<\vol(A|_F)=\vol(K_F) \leq v$. We are done by Theorem \ref{g-pair effective canonical bundle formula}.

   \item We now shortly explain the general $\bR$-divisor case. When $\Phi$ is also a finite set, we use Lemma \ref{Decomposition of g-div} to decompose $K_X+B+M$ uniformly 
   $$K_X+B+M=\sum_{i=1}^lr_i(K_X+B_i+M_i)$$
   such that each $(X, B_i+M_i)$ satisfy the conditions listed in Lemma \ref{Decomposition of g-div}. In particular, we can apply the $\bQ$-divisor version proved above to find an effective adjunction formula for $(X,B_i+M_i)\to Z$, and thus also an effective adjunction formula for $(X,B+M)\to Z$ as well.

   Now when $\Phi$ is a DCC set, by \cite[Theorem 1.5]{effectiveIitaka}, we know that $B_Z\in \Psi$ for some fixed DCC set $\Psi$ depending only on $d,\Phi,\Phi'$. So we are mainly focused on the moduli part $M_{Z'}'$ on some high resolution $\ZZ \to Z$. Note that $K_F+B_F+M_F\sim_{\bR}0$ implies the coefficients of $B_F$ belong to a fixed finite set, by \cite[Theorem 1.6]{effectiveIitaka}. In particular, the coefficients of the horizontal/$Z$ part $B^h$ of $B$ belong to a finite set. By a similar treatment as the proof of Theorem \ref{effective adjunction formula}, we may construct a new g-pair $(W, \Delta+M_W)$ on some log resolution $W\to X$ and then run some MMP to reach a model $(X'',\Delta''+M'')\to Z''$ such that the coefficients of $\Delta''$ belong to a fixed finite set, and $K_{X''}+B''+M''\sim_{\bR}0/Z''$. Moreover, the pullback of $K_X+B+M$ and $K_{X''}+B''+M''$ on a common resolution are the same over the generic point of $Z''$. We may replace $(X,B+M)\to Z$ by $(X'',\Delta''+M'')\to \ZZ$ and then we are done by the previous paragraph.
\end{enumerate}
\end{proof}

\begin{proof}[Proof of Corollary \ref{MAIN COR1} and Corollary \ref{MAIN COR2}]
 These are the direct consequence of Lemma \ref{g-pair Fanotype eff adjunction} and \cite[Theorem 1.3]{BVGP}.   
\end{proof}

\subsection{Proof of Theorem \ref{MAINTHMnew}}

We now deal with the case when $(X,B+M)$ is g-klt. More generally, we can assume $(X,B+M)\to Z$ is generalised klt over the generic point of $Z$. In this case, we may also derive an effective adjunction formula, where the treatment is similar to the proof of \cite[Theorem 11.1]{singularitiesonFanofibrations}. Here the interesting thing is, although we cannot make $p(K_X+B+M)\sim pf^*(K_Z+B_Z+M_Z)$ in general, we can still control $M_Z$ so that $pM_{Z'}$ is Cartier nef on any high resolution $Z'\to Z$.

\begin{prop}\label{effectiveadjunctionwhensingularitiesarecontrolled}
Let $d,q$ be two positive integers, $u\in \bQ^{>0}$ a positive rational number, and $\Phi\subset \bQ^{\geq0}$ a DCC set.  Consider the set of $(X,B+M),A$ such that \begin{itemize}     
\item $(X,B+M)$ is a projective generalised lc pair of dimension $d$ with data $\pi: \XX \rightarrow X$ and nef part $\MM$,     
\item $B\in \Phi$ and $q\MM$ is nef Cartier,     
\item there is a contraction $f:X\to Z$ such that $K_X+B+M\sim_\bQ 0/Z$, 
\item $(X,B+M)$ is generalised klt over the generic point $\eta_Z$ of $Z$, and
\item there is an integral divisor $A$ on $X$ such that $0<\vol (A|_F)\leq u$, where $F$ is the general fiber of $f:X\rightarrow Z$. 
\end{itemize}  

Then there is a positive integer $p$ depending only on $d,q,u,\Phi$ such that there is an adjunction formula $$K_X+B+M\sim_{\bQ} f^*(K_Z+B_Z+M_Z)$$ and $pM_{Z'}$ is nef Cartier for any high resolution $\ZZ \to Z$.
    
\end{prop}

\begin{proof}
 \begin{enumerate} [label=\textsl{Step} \arabic{enumi}.,wide=13pt,itemsep=13pt]  
 \item We first reduce to the case where $(X,B+M)$ is generalised $\epsilon$-lc for some fixed $\epsilon>0$, and $\dim Z=1$. The argument is similar to the proof of Theorem \ref{g-pair effective canonical bundle formula}. After replacing $X$ with a $\bQ$-factorialisation, we may assume $X$ is $\bQ$-factorial. First of all, on the general fiber $F$ of $f$, since $(F, B_F+M_F)$ is g-klt Calabi-Yau pair, by Lemma \ref{g-pair version of 2.48}, there is some rational number $\epsilon>0$ depending only on $\dim F,\Phi $ such that $(F,B_F+M_F)$ is generalised $\epsilon$-lc. In particular, for any prime divisor $P$ over $X$ that is horizontal over $Z$, we have $$a(P,X,B+M)\geq \epsilon.$$

 Now as in the proof of Theorem \ref{g-pair effective canonical bundle formula}, we first take a high resolution $\bar{Z}\to Z$ and a log resolution $\phi: \bar{X}\to X$ of $(X,\Supp B)$ such that $\bar{X}\to \bar{Z}$ is a morphism and $M'$ descends to $\bar{X}$. Take $\bar{\Delta}$ to be the birational transform of the horizontal$/Z$ part of $B$ plus $(1-\epsilon)$-times the horizontal$/Z$ part of the reduced exceptional divisor of $\phi$. It's easy to see that $(\bar{X},\bar{\Delta}+\bar{M})$ is generalised $\epsilon$-lc. Then as before, we may run two consecutive MMP on $(K_{\bar{X}}+\bar{\Delta}+\bar{M})$, first over $X$, then over $\bar{Z}$ that ends with a good minimal model $W$. We obtain a contraction $W\to \ZZZ/\bar{Z}$. By construction, the moduli divisor of $(W,B_W+M_W)\to \ZZZ$ coincides with the moduli divisor of $(X,B+M)\to Z$, as the MMP process doesn't change the pair over the generic point of $\ZZZ$. Let $A_W$ be the birational transform of the horizontal$/Z$ part of $A$. Moreover, it suffice to find a bounded $p$ such that $pM_{\ZZZ}$ is integral, so we may cut down the dimension of $\ZZZ$ by general hyperplane sections and assume $\dim Z=1$. Replace $(X,B+M), A$ by the restriction of $(W,B_W+M_W),A_W$, and replace $\Phi$ by $\Phi\cup\{1-\epsilon\}$, we may reduce to the case where $(X,B+M)$ is $\bQ$-factorial generalised $\epsilon$-lc, and $\dim Z=1$.

 \item In this step we make some further reductions. As $Z$ is a curve, it suffice to find $p$ such that $pM_Z$ is integral. Since $(X,0)$ is $\epsilon$-lc, applying \cite[Theorem 1.1]{GOPV} to the minimal model of $A$ over $Z$, after replacing $A$ by a bounded multiple, we may assume $A$ is effective.

 First assume $B+M\equiv 0$ over the generic point of $Z$. In this case $K_Z \sim_\bQ 0$ over the generic point, thus we have $B$ is vertical over $Z$, and $M\sim_\bQ 0$ over the generic point of $Z$. As $Z$ is a curve, by the same argument as Step 3 of the proof of Lemma \ref{curve case}, we see $M\sim_\bQ 0/Z$ and thus $K_X+B\sim_\bQ 0/Z$. We then can apply \cite[Lemma 7.4]{BVGP} such that there is a bounded $p$ and an effective adjunction formula for $(X,B)$:
 $$p(K_X+B)\sim pf^*(K_Z+B_Z+M^1_Z)$$
 such that $pM^1_{Z'}$ is integral. Moreover, by the proof of \cite[Lemma 7.3]{BVGP}, we may assume the multiplicities of fibers of $f:X\to Z$ is bounded from above. Let $$qM=f^*L_Z+\sum_{i\in I}a_i\Div(s_i),$$ where $a_i\in \bQ$ and $\Div(s_i)$ are principle divisors with no common support. Since $qM$ is integral, any $a_i\Div(s_i)$ that is horizontal$/Z$ must be an integral divisor, and any $a_i\Div(s_i)$ that is vertical$/Z$ must be the pullback of some $b_i\Div(t_i)$ on $Z$. In particular, by decomposing $\sum_{i\in I}a_i\Div(s_i)$ into the horizontal part and the vertical part, we have
 $$qM= f^*L_Z+\sum_{i\in I_1}a_i\Div(s_i) +\sum_{i\in I_2}f^*(b_i\Div(t_i)).$$
Let $M^2_Z=\frac{1}{q}(L_Z+\sum_{i\in I_2}b_i\Div(t_i))$, then $$qf^*M^2_Z=qM-\sum_{i\in I_1}a_i\Div(s_i)\sim_{\bQ} qM$$
is an integral divisor, so after replacing $p$ by a bounded multiple, the coefficient $p\mu_zM^2_Z$ is integral for any closed point $z\in Z$. Alternatively, we may also apply \cite[Lemma 3.5]{GOPV} and \cite[Theorem 11.1]{singularitiesonFanofibrations} on $qM^2_Z$ to find a bounded $p$ such that $pM^2_Z$ is integral. This implies $pM_Z=p(M^1_Z+M^2_Z)$ is integral indeed.

 \item From now on assume $B+M$ is not numerically trivial over the generic point of $Z$. In this case $K_X$ is not pseudo-effective over $Z$. Replacing $X$ by the minimal model of $A$ over $Z$, we may assume $A$ is nef and big. Let $t$ be the smallest number such that $K_X+tA$ is pseudo-effective over $Z$. Clearly $t>0$. By the argument in \cite[Theorem 11.1]{singularitiesonFanofibrations}, we can reduce to the case that $t\geq 1$ is a rational number, and such $t$ has only finitely many choices depending only on $d,\epsilon$. Note that these reductions relies on \cite[Lemma 4.11]{GOPV}, which requires restriction on the singularities of $X$.

 View $(X, tA)$ as a generalised pair with nef part $tA$ itself. Clearly $(X,tA)$ is generalised $\epsilon$-lc. Let $X''$ be the good minimal model of $K_X+tA$ over $Z$, which exists by \cite[Lemma 4.4]{effectiveIitaka} as $tA$ is big over $Z$. Let $B'',M''$ be the pushdown of $B,M$ to $X''$.  Note that $K_{X''}+tA''$ is semi-ample but not big over $Z$. Both $(X'',tA'')$ and $(X'',B''+M'')$ are generalised $\epsilon$-lc by construction. We have a non-birational contraction $f'':X''\to Y/Z$. Moreover, as $A''$ is big over $Y$, we see $-K_{X''}$ is big over $Y$ thus $X''\to Y/Z$ is of Fano type. In particular, by Lemma \ref{g-pair Fanotype eff adjunction} and its proof, there is a bounded $p$ such that we have the following adjunction formulas:

 $$p(K_{X''}+tA'')\sim pf{''}^*(K_Y+C_Y+N_Y)$$

 and

 $$p(K_{X''}+B''+M'')\sim pf''^*(K_Y+B_Y+M_Y) $$

 where both $B_Y$ and $C_Y$ belong to a fixed DCC set $\Psi$, and $pM_{Y'}$ is Cartier nef for any high resolution $Y'\to Y$.

 \item By \cite[Theorem 9.3]{singularitiesonFanofibrations}, the multiplicities of the fibers of $f''$ over codimension one points are bounded, we can assume $p(K_Y+C_Y+N_Y)$ is also integral, after replacing $p$ by a bounded multiple. Let $$A_Y:=p(K_Y+C_Y+N_Y).$$  It's easy to see $A_Y$ is ample over $Z$, and by the proof of \cite[Theorem 11.1]{singularitiesonFanofibrations}, we can see $\vol(A_Y|_{F_Y})\leq (1+t)u$ is bounded from above, where $F_Y$ is the general fiber of $Y\to Z$.

 By \cite[Theorem 11.1]{singularitiesonFanofibrations}, there is a fixed $\tau>0$ such that $(Y,B_Y+M_Y)$ is generalised $\tau$-lc. Replace $(X,B+M),A$ by $(Y,B_Y+M_Y),A_Y$, and replace $\epsilon,q, \Phi,u$ by $\tau, p, \Psi, (1+t)u$, we have $(Y,B_Y+M_Y)\to Z$ with $K_Y+B_Y+M_Y\sim _\bQ 0/Z$.

 \item As $\dim Y<\dim X$, by induction, continuing the process above will eventually give a birational morphism $Z'\to Z$ and a generalised pair $(Z',B_{Z'}+M_{Z'})$. Moreover, there is a bounded $p$ such that $pM_{Z'}$ is integral. This implies $pM_Z$ is integral as well. 
 \end{enumerate}
\end{proof}

\begin{remark}
    As before, Proposition \ref{effectiveadjunctionwhensingularitiesarecontrolled} also holds for $\bR$-divisors. Indeed, Lemma \ref{g-pair Fanotype eff adjunction} is stated for $\bR$-divisors, so the proof above goes through with minor changes. We omit the detail of proof for its similarity to the other results in this paper.
\end{remark}

\begin{proof}[Proof of Theorem \ref{MAINTHMnew}] This is the direct consequence of Proposition \ref{effectiveadjunctionwhensingularitiesarecontrolled} and \cite[Theorem 1.3]{BVGP}.   
\end{proof}

\section{Boundedness of generalised klt trivial fibrations}
In this section we prove several boundedness results about generalised klt trivial fibrations. They are the direct application of the results in previous sections.

\subsection{Boundedness of polarised g-klt Calabi-Yau pairs}
We first provide a generalised version of \cite[Corollary 1.8]{GOPV}

\begin{thm}
  Let $d\in \bN$, $v\in \bR^{>0}$, and $\Phi\subset \bR^{\geq 0}$ be a finite set. Consider $(X,B)$ and $N$ such that
  \begin{itemize}
      \item $(X,B+M)$ is a projective generalised klt pair of dimension $d$ with data $\pi:X'\to X$ and $M'$,
      \item $B\in \Phi$, $M'=\sum \mu_j M_j'$ where $M'_j$ is nef Cartier and $\mu_j\in \Phi$ for any $j$,
      \item $N\geq 0$ is an ample $\bR$-divisor such that $N\in \Phi$, and
      \item $\vol(N)\leq v$.
  \end{itemize}
  Then the set $(X,\Supp(B+N))$ forms a bounded family. 
\end{thm}

\begin{proof}
   By Remark \ref{remark about gklt}, $(X, B+M)$ is g-$\epsilon$-lc for some $\epsilon>0$ depending only on $d, \Phi$, and there is $\lambda>0$ depending only on $d,v, \Phi$ such that $(X, B+\lambda N+M)$ is g-$\frac{\epsilon}{2}$-lc. Moreover, replacing $\Phi$ by $\Phi \cup\lambda\Phi$, we may assume $B+\lambda N\in \Phi$, and 
   $$\vol(K_X+B+\lambda N+M)=\vol(\lambda N)\leq \lambda^dv$$
   Thus $(X, B+\lambda N+M)\in \cF_{gklt}(d,\Phi, \leq v,\frac{\epsilon}{2})$. By Theorem \ref{bdd gpair}, we see such $(X, B+\lambda N+M)$ forms a bounded family, in particular $(X,\Supp(B+N))$ forms a bounded set of couples. 
\end{proof}

\begin{remark}
    We remark that if we require $\vol(N)=v$ to be fixed, then the coefficient set $\Phi$ can be a DCC set, by Theorem \ref{bdd gpair} as well.
\end{remark}

\subsection{Proof of Corollary \ref{MAIN COR3}}

\begin{proof}[Proof of Corollary \ref{MAIN COR3}]
  First of all, by \cite[Theorem 1.8]{CHL24}, we see the Iitaka volume of $-(K_X+B)$ belongs to a fixed ACC set depending only on $d,\epsilon,v,\Phi$. Note that we have an effective adjunction for $(X,B)\to Z$:
  $$q(K_X+B)\sim qf^*(K_Z+B^1_Z+M^1_Z)$$
  for some bounded $q\in \bN$. Let $M:=-2(K_X+B)$, and view $(X,B+M)$ as a g-pair with nef part $M$ itself. Note that $M$ can be written as $\frac{2}{l}(-l(K_X+B))$ where $-l(K_X+B)$ is nef Cartier.

  Note that 
  $$q(K_X+B+M)= -q(K_X+B)\sim 0/Z,$$
  we automatically have an adjunction formula for $(X,B+M)\to Z$

  $$q(K_X+B+M)\sim qf^*(K_Z+B_Z^2+M_Z^2).$$
  Thus by applying Theorem \ref{MAIN THM2}, we see $\Ivol(K_X+B+M)=\Ivol (-(K_X+B))$ also belongs to a DCC set.
  In particular, there is a finite set $J$ depending only on $d,\epsilon,v,\Phi$ such that the Iitaka volume of $-(K_X+B)$ belongs to $J$.
\end{proof}

\begin{remark}\label{ellpitic surfaces}
    We remark that the last condition in Corollary \ref{MAIN COR3} can not be omitted, even if we consider $(X,0)$ with $X$ being smooth. Indeed, consider $f: X\to \bP^1$ be a smooth elliptic surface, then by \cite[Proposition 5.4]{CHL24}, we have $\Ivol(-K_X)\in \{\frac{1}{m}|m\in \bN\}$, and every $\frac{1}{m}$ can be realized as $\Ivol(-K_X)$ for some rational elliptic surface $X$. In particular, such set of Iitaka volumes is only ACC but not finite.
\end{remark}

\subsection{Boundedness of g-klt trivial fibrations}

We first prove the following result controlling the singularity of certain g-klt pairs. This result can be viewed as a partial generalization of \cite[Theorem 1.6]{Zhu25} to the g-pair setting.

\begin{prop}\label{control singularities for g-pair}
  Let $d,q \in \bN$, $u,v \in \bQ^{>0}$, and $\Phi \subset \bQ ^{\geq 0}$ a DCC set. Let $\cF_{gklt}(d,q,\leq u,v, \Phi)$ be the set of g-pairs $(X, B+M), A$ with data $\pi: \XX \to X$ and $M'$ such that
  \begin{itemize}
      \item $(X, B+M)$ is generalised klt of dimension $d$,   \item the coefficients of $B$ are in $\Phi$,    
      \item $qM'$ is nef Cartier,    
      \item $f: X \to Z$ is a contraction with $K_X + B +M \sim_{\mathbb{Q}} 0/Z$,    
      \item $\kappa(K_X + B +M) = \dim Z$, 
      \item $\Ivol(K_X+B+M)=v$,
      \item we have a generalised adjunction formula    
      $$q(K_X+B+M)\sim qf^*(K_Z+B_Z+M_Z), $$    
      \item there is an effective $\bQ$-divisor $A \geq 0$ on $X$ such that $A\in \Phi$ and over some non-empty open subset of $Z$: $(X, B + tA +M)$ is generalised lc for some $t > 0$ and $A$ is semi-ample, and   
      \item $0<\operatorname{vol}(A|_F) \leq u$ for the general fibres $F$ of $f$.
  \end{itemize}

  There exists $\epsilon>0$ depending only on $d,q,u,v,\Phi$ that if $(X,B+M),A\in \cF_{gklt}(d,q,\leq u,v, \Phi)$, $(X,B+M)$ is generalised $\epsilon$-lc.
\end{prop}

\begin{proof}
    \begin{enumerate} [label=\textsl{Step} \arabic{enumi}.,wide=13pt,itemsep=13pt]
    \item By Theorem \ref{g-pair effective canonical bundle formula}, we see there exists some $p\in \bN$ depending only on $d,q,u,\Phi$ such that $pM'_{Z'}$ is Cartier nef for some high resolution $\ZZ \to Z$. Moreover, by \cite[Theorem 1.5]{effectiveIitaka}, the coefficients of $B_Z$ belong to a fixed DCC set $\Psi$ depending only on $d, q,\Phi$. We may assume $\frac{1}{p}\in \Psi$. Since
    $$\vol(K_Z+B_Z+M_Z)=\Ivol(K_X+B+M)=v,$$
    we see that $(Z,B_Z+M_Z)\in \cF_{gklt}(\dim Z,\Psi,v)$. By \cite[Theorem 1.5]{BVGP}, $(Z, B_Z+M_Z)$ is generalised $\delta$-lc, for some $\delta>0$ depending only on $\dim Z, \Psi,v$, thus depending only on $d,q,u,v,\Phi$.
    \item When $P$ is a prime divisor over $X$ that is horizontal over $Z$, it determines a prime divisor $S$ over the general fiber $F$ of $f$. Since $(F,B_F+M_F)$ is a g-klt Calabi-Yau g-pair, with $B_F\in \Phi$ and $qM_{F'}'$ is nef Cartier for some resolution $F'\to F$, by Lemma \ref{g-pair version of 2.48} we see that there is some $\tau>0$ depending only on $\dim F, \Phi$ that $(F,B_F+M_F)$ is g-$\tau$-lc. This implies that $a(P,X,B+M)=a(S,F,B_F+M_F)\geq \tau$.

    \item When $P$ is a prime divisor over $X$ that is vertical over $Z$, consider some high resolution such that $P$ is a prime divisor on $X'$ and its image on $Z'$ is a prime divisor $E$
    \begin{displaymath}
			\xymatrix{
				\XX \ar[d]_-{f^\prime} \ar[r]^\pi  & X \ar[d]^f            \\
				\ZZ \ar[r] ^\mu & Z   }
		\end{displaymath}
	Let 
		\begin{equation*}
			K_{\XX}+\BB+M'=\pi^*(K_X+B+M) 
		\end{equation*}
		and
		\begin{equation*}
			K_{\ZZ}+B_{\ZZ}+M_{\ZZ}=\mu^*(K_Z+B_Z+M_Z).
		\end{equation*}
		Since $(Z,B_Z+M_Z)$ is generalized $\delta$-lc,
		\begin{equation*}
			a(E,\ZZ,B_{\ZZ}+M_{\ZZ})=a(E,Z,B_Z+M_Z)\geq \delta.
		\end{equation*}
		Therefore, $$\mult_E B_{\ZZ}\leq 1-\delta.$$ By the definition of discriminant divisors, $(\XX,\BB+\delta f^{\prime *}E)$ is sub-lc over the generic point of $E$. This implies that $$\mult_P \BB+\delta \mult_P f^{\prime *}E\leq 1$$ and hence $\mult_P \BB\leq 1-\delta$ as $E$ is a Cartier divisor on $Z'$. Thus 
		\begin{equation*}
			a(P,X,B)=a(P,\XX,\BB)\geq \delta.
		\end{equation*}	
		\item  From the above arguments we see that $(X,B+M)$ is g-$\epsilon$-lc, where $\epsilon:=\min\{\tau, \delta\}$ depending only on $d, q, \Phi, u, v$ indeed.
    \end{enumerate}
\end{proof}

\begin{proof}[Proof of Theorem \ref{MAINTHM3}]
 By Theorem \ref{g-pair effective canonical bundle formula}, we see there exists some $p\in \bN$ depending only on $d,q,u,\Phi$ such that $pM'_{Z'}$ is Cartier nef for some high resolution $\ZZ \to Z$. Moreover, by \cite[Theorem 1.5]{effectiveIitaka}, the coefficients of $B_Z$ belong to a fixed DCC set $\Psi$ depending only on $d, q,\Phi$. We may assume $\frac{1}{p}\in \Psi$. Since    $$\vol(K_Z+B_Z+M_Z)=\Ivol(K_X+B+M)=v,$$    we see that $(Z,B_Z+M_Z)\in \cF_{gklt}(\dim Z,\Psi,v)$. By \cite[Theorem 1.5]{BVGP}, $\cP:=\cF_{gklt}(\dim Z,\Psi,v)$ is a bounded family of g-pairs, depending only on $\dim Z,\Psi,v$ thus depending only on $d,q,u,v,\Phi$.

 The existence of $\epsilon>0$ is obtained by Proposition \ref{control singularities for g-pair}. It suffice to find a bounded $l$ such that $l(K_Z+B_Z+M_Z)$ is very ample. Note that as $(Z, B_Z+M_Z)\in \cF_{gklt}(\dim Z,\Psi,v)$, by \cite[Theorem 1.5]{BVGP}, the coefficients of $B_Z$ actually belong to a fixed finite set, thus $l(K_Z+B_Z+M_Z)$ is integral for some bounded $l$. As $\cP$ is a bounded family, after replacing $l$ by a bounded multiple, we may assume $l(K_Z+B_Z+M_Z)$ is ample Cartier.
Now by \cite[Lemma 2.15]{Zhu25}, after replacing $l$ by a bounded multiple again, we may assume $l(K_Z+B_Z+M_Z)$ is very ample, where $l$ depends only on $d,u,v,q,\Phi$.
\end{proof}

Finally we state and prove an $\bR$-divisor version of Corollary \ref{MAINCOR4}.

\begin{prop}
  Let $d,\in \bN$, $v\in \bR^{>0}$, $\Phi\subset \bR^{\geq 0}$ a DCC set, and $\Phi'\subset \bR^{\geq 0}$ a finite set. Consider the set of $(X,B+M)$ such that  \begin{itemize}      
  \item $(X,B+M)$ is a projective generalised klt pair of dimension $d$ with data $\pi: \XX \rightarrow X$ and nef part $\MM$,      
  \item $B\in \Phi$ and $\MM=\sum \mu_jM'_j$ where $\mu_j\in \Phi'$, $M'_j$ is nef Cartier for any $j$,     
  \item $K_X+B+M$ is semi-ample defining a contraction $f:X\rightarrow Z$,      
  \item $\mathrm{Ivol}(K_X+B+M)=v$, and     
  \item $X$ is of Fano type over $Z$.  
  \end{itemize}  
  
  Then there is a bounded family of g-klt pairs $\cP$, and a positive number $\epsilon \in \bR^{>0}$ depending only on $d,,v,\Phi,\Phi'$ such that $(Z,B_Z+M_Z)$ belongs to $\cP$, and $(X, B+M)$ is generalised $\epsilon$-lc. Moreover, the set of such $X$ is also bounded.  
\end{prop}

\begin{proof}
    The treatment is similar to the proof of Theorem \ref{MAINTHM3}. We use Lemma \ref{g-pair Fanotype eff adjunction} instead of Theorem \ref{g-pair effective canonical bundle formula}. We see $(Z, B_Z+M_Z)\in \cF_{gklt}(\dim Z,v, \Psi)$, which forms a bounded family of g-pairs. The existence of $\epsilon>0$ also follows from Proposition \ref{control singularities for g-pair}, where the proof is almost the same for $\bR$-divisors.

    So we only need to show $X$ also forms a bounded family. As $\cF_{gklt}(\dim Z,v, \Psi)$ is a bounded set, there is a very ample divisor $H$ on $X$ and a bounded $r\in \bN$, such that $H^{\dim Z}\leq r$ and $(K_Z+B_Z+M_Z)\cdot H^{\dim Z-1}\leq r$. Then by Lemma \ref{2.17}, we see there is a bounded $l\in \bN$ such that $lH-(K_Z+B_Z+M_Z)$ is ample. Since $(X,B+M)$ is generalised $\epsilon$-lc and $B\in \Phi$ is a DCC set, we have $B\geq c$ for some fixed $c>0$. By \cite[Theorem 2.3]{FTF}, such $X$ indeed forms a bounded family.

\end{proof}

\section{Examples and further discussions}

\begin{example}
    The restriction on $\vol(A|_F)\leq u$ in Definition \ref{DEF2} cannot be removed. Indeed, as in Remark \ref{ellpitic surfaces}, if we take $B=0$ and $M=-2K_X$ there, the set of Iitaka volumes $\Ivol(K_X+B+M)=\Ivol(-K_X)$ is not DCC.
\end{example}

\begin{example}
    The assumption $\kappa(K_X+B+M)=\dim Z$ in Definition \ref{DEF2} cannot be removed. Indeed, let $Y\to \bP^1$ be a rational elliptic surface as in Remark \ref{ellpitic surfaces}, and let $C$ be a fixed elliptic curve. Then let $X:=Y\times C$ and $f:X\to Y$ be the projection to $Y$. We have $K_X=f^*(K_Y)$ and if $A$ is a section of $X$ over $Y$, $\vol(A|_{C\times\{y\}})=1$ is fixed. Let $B=0$ and $M=-2K_X$, and $\Ivol(K_X+B+M)=\Ivol(-K_X)=\Ivol(-K_Y)$ which is not DCC by the previous example.
\end{example}

\begin{example}
    It's important to view $(X,B+M)$ as a g-pair itself, not just a g-pair over $Z$. That is, we require $M'$ to be actually nef, not just nef over $Z$. Otherwise, the set of Iitaka volumes may behave very strangely.

    The following example comes from \cite[Example 5.3]{CHL24}. Take $Y=\bP(p,q,r)$ such that $Y$ is well-formed. Let $Z\to Y$ be the minimal resolution of $Y$. Take $X:=Z\times C$ where $C$ is an elliptic curve, let $f:X\to Z$ be the projection to $Z$, and $A$ a section of $X$ over $Z$. Then $X$ is smooth and $\vol(A|_{C\times\{z\}})=1$ for any $z\in Z$, but if we take $B=0$ and $M=-2K_X=2f^*(-K_Z)$ is not nef globally, but still nef over $Z$. Then $$\Ivol(K_X+B+M)=\Ivol(-K_X)=\vol(-K_Z)=\vol(-K_Y) = \frac{(p+q+r)^2}{pqr}$$
    forms a dense subset in $\bR^{>0}$ as $p,q,r$ vary.
\end{example}

\begin{example}
    One thing that makes the g-pair case difficult is we don't have similar result as \cite[Lemma 7.4]{BVGP}. That is, even we have $\vol(A|_F)\leq u$ on the general fiber, we cannot find a bounded $q\in \bN$ such that $q(K_X+B+M)\sim qf^*L_Z$ for some $\bQ$-divisor $L_Z$ in general. 

    For example, take $Z=\bP^1$ and $C$ a fixed elliptic curve. Let $X=Z\times C$, denote $f:X\to Z$ and $g:X\to C$ to be the corresponding projection. As before, take $A$ to be a section of $X\to Z$. Let $B=0$ and $M:=g^*R$ where $R$ is a torsion divisor on $C$ with torsion index $m$, that is, $mR\sim 0$ but $lR$ is not linearly equivalent to $0$ for any $l<m$. Such torsion divisor always exists. Now $K_X+B+M=f^*K_Z+g^*R$ and $K_X+B+M\sim_\bQ0/Z$. For any $q\in \bN$ such that $q(K_X+B+M)\sim qf^*(K_Z+B_Z+M_Z)$, restricting to the general fiber $F\cong C$ of $f$ yields $q(K_F+B_F+M_F)=qR\sim 0$. Thus $m|q$ and as $m$ can tend to infinity, there is no such bounded $q$.
\end{example}

However, as Iitaka volume is a numerical invariant, it is expected that the assumption on $q(K_X+B+M)\sim qf^*(K_Z+B_Z+M_Z)$ may not be necessary. Theorem \ref{MAINTHMnew} is one direction, where we restrict ourselves to g-klt pairs. For another direction, the following result shows that at least for generalised lc surfaces, we don't require such an additional assumption.

\begin{prop}
    Let $\cI_{glc}(2,\Phi,q,u)$ be a set of projective generalised surface pairs $(X,B+M)$ such that

    \begin{itemize}
        \item $(X,B+M)$ is generalised lc of dimension $2$,
        \item $B\in \Phi$,
        \item $qM$ is nef Cartier,
        \item $f:X\to Z$ is a contraction to a curve $Z$ with $K_X+B+M\sim_\bQ 0/Z$
        \item there is an integral divisor $A$ on $X$ that is ample over $Z$ such that $\vol(A|_F)=u$ on the general fiber $F$ of $f$
    \end{itemize}

    If $u\in \bQ^{>0}$ and $\Phi\subset \bQ^{\geq 0}$ is a DCC set, then 
    $$\{\Ivol(K_X+B+M)|(X,B+M)\in \cI_{glc}(2,\Phi,q,u)\}$$
    satisfy the DCC property.
\end{prop}

\begin{proof}
    On the general fiber $F$, we have $K_F+B_F+M_F\sim _\bQ 0$, note that $(F, B_F+M_F)$ is a g-lc curve, $F$ is actually smooth, and $g(F)\leq 1$. If $F=\bP^1$, then $X$ is of Fano type over $Z$ and everything is easy. So we may assume $F$ is an elliptic curve, thus $K_F=B_F=0$. In particular, $M_F\sim _{\bQ}0/Z$. Thus $M\sim_\bQ  P$ where $P$ is a vertical/$Z$ divisor. We claim that $M\sim_\bQ 0/Z$. By adding some negative $\bQ$-linear combination of the fibers $\sum b_i F_i$, we may assume $M\sim P':=P+\sum b_i F_i/ Z$ and $P'\leq 0$. But then $P'$ must be $0$, otherwise since $P'$ is nef over $Z$, there is some component $C$ of some fiber that intersects with $\Supp P'$ but not contained in $\Supp P'$. This yields a contradiction as $(P'\cdot C)<0$. Thus we have $K_X+B\sim_\bQ 0/Z$ and $M\sim_\bQ 0/Z$.

    It suffice to show that if $M\equiv \alpha f^* z$ where $z$ is a general point of $Z$, then $\alpha$ belongs to some fixed DCC set. Indeed, as then $\Ivol(K_X+B+M)=\deg (K_Z+B_Z+M_Z)= 2g(Z)-2+\deg (B_Z) +\alpha$ will belong to a DCC set.

    Note that $\vol(A|_F)=u$ implies $(A\cdot f^*z)=u$ where $z\in Z$ is a general point. Then $(qM\cdot A)= q\alpha u$, where $(qM\cdot A)\in \bZ^{>0}$ as $qM$ is Cartier. This implies $\alpha\in \frac{1}{qu}\bZ^{>0}$ belongs to a DCC set indeed.
\end{proof}

\hfill

\providecommand{\bysame}{\leavevmode\hbox to3em{\hrulefill}\thinspace}
\providecommand{\MR}{\relax\ifhmode\unskip\space\fi MR }
\providecommand{\MRhref}[2]{%
	\href{http://www.ams.org/mathscinet-getitem?mr=#1}{#2}
}
\providecommand{\href}[2]{#2}


\begin{thebibliography}{HMX18}
	
	\bibitem[Amb99]{ambroadjunctionformula}
	F.~Ambro, \emph{{The Adjunction Conjecture and its applications}},
	arXiv:math/9903060v3 (1999).
	
	\bibitem[Amb05]{ambromoduli}
	Florin Ambro, \emph{The moduli {$b$}-divisor of an lc-trivial fibration},
	Compos. Math. \textbf{141} (2005), no.~2, 385--403.

    \bibitem[BCHM10]{BCHM10}
	C.~Birkar, P.~Cascini, C.~D.~Hacon, and J.~McKernan, \emph{{Existence of minimal models for varieties of log general type}}, J. Amer. Math. Soc. \textbf{23} (2010), no.~2, 405--468.
	
	\bibitem[BH22]{VGP}
	C.~Birkar and C.~D. Hacon, \emph{Variations of generalised pairs},
	arXiv:2204.10456 (2022).
	
	\bibitem[Bir12]{Bir12}
	C.~Birkar, \emph{Existence of log canonical flips and a special {LMMP}}, Pub.
	Math. IHES. \textbf{115} (2012), 325--368.
	
	\bibitem[Bir19]{BABI}
	C.~Birkar, \emph{{Anti-pluricanonical systems on Fano varieties}}, Ann. of Math.
	(2) \textbf{190} (2019), no.~2, 345 -- 463.
	
	\bibitem[Bir21a]{BVGP}
	C.~Birkar, \emph{Boundedness and volume of generalised pairs}, arXiv:2103.14935
	(2021).
	
	\bibitem[Bir21b]{BABII}
	C.~Birkar, \emph{{Singularities of linear systems and boundedness of Fano
			varieties}}, Ann. of Math. (2) \textbf{193} (2021), no.~2, 347--405.
	
	\bibitem[Bir22]{moduliofalgebraicvarieties}
	C.~Birkar, \emph{Moduli of algebraic varieties}, arXiv:2211.11237 (2022).
	
	\bibitem[Bir23a]{GOPV}
	C.~Birkar, \emph{Geometry of polarised varieties}, Pub. Math. IHES. \textbf{137}
	(2023), 47–105.
	
	\bibitem[Bir23b]{singularitiesonFanofibrations}
	C.~Birkar, \emph{{Singularities on Fano fibrations and beyond}}, arXiv:2305.18770
	(2023).
	
	\bibitem[Bir24]{FTF}
	C.~Birkar, \emph{{Boundedness of Fano type fibrations}}, Ann. Sci. \'Ec. Norm.
	Sup\'er. (4) \textbf{57} (2024), no.~3, 787--840.
	
	\bibitem[BZ16]{effectiveIitaka}
	C.~Birkar and D-Q. Zhang, \emph{{Effectivity of Iitaka fibrations and
			pluricanonical systems of polarized pairs}}, Pub. Math. IHES. \textbf{123}
	(2016), no.~1, 283--331.
	
	\bibitem[Che23]{CGD20}
	G.~Chen, \emph{Boundedness of $n$-complements for generalized pairs}, Eur. J.
	Math. \textbf{9} (2023), no.~4, 95.
	
	\bibitem[CHL23]{CHL23}
	G.~Chen, J.~Han, and J.~Liu, \emph{On effective log {Iitaka} fibrations and
		existence of complements}, Int. Math. Res, Not. \textbf{2024} (2023), no.~10,
	8329--8349.

    \bibitem[CHL24]{CHL24}
	G.~Chen, J.~Han, and W.~Liu, \emph{{On the litaka volumes of log canonical surfaces and threefolds}}, arXiv:2407.07391 (2024).
	
	\bibitem[Cho08]{invariantIitakadimension}
	S.~R. Choi, \emph{The geography of log models and its applications}, Ph.D.
	thesis, Johns Hopkins University, 2008.
	
	
	
	\bibitem[Has19]{Hashizume19}
	K.~Hashizume, \emph{Remarks on special kinds of the relative log minimal model
		program}, Manuscripta Math. \textbf{160} (2019), no.~3-4, 285--314.
	
	
	
	
	\bibitem[HLX23]{effectiveadjunctionformulawithrealcoefficients}
	J.~Han, J.~Liu, and Q.~Xue, \emph{{On the equivalence between the effective
			adjunction conjectures of Prokhorov-Shokurov and of Li}}, arXiv:2312.15397,
	to appear in Algebra Number Theory (2023).
	
	\bibitem[HMX13]{birationalatuomorphism}
	C.~D. Hacon, J.~M\textsuperscript{c}Kernan, and C.~Xu, \emph{On the birational
		automorphisms of varieties of general type}, Ann. of Math. (2) \textbf{177} (2013), no.~3,
	1077--1111.
	
	\bibitem[HMX14]{ACCLCT}
	C.~D. Hacon, J.~M\textsuperscript{c}Kernan, and C.~Xu, \emph{{ACC for log
			canonical thresholds}}, Ann. of Math. (2) \textbf{180} (2014), no.~2,
	523--571.
	
	\bibitem[HMX18]{HMX18}
	C.~D. Hacon, J.~M\textsuperscript{c}Kernan, and C.~Xu, \emph{Boundedness of
		moduli of varieties of general type}, J. Eur. Math. Soc. \textbf{20} (2018),
	no.~4, 865--901.
	
	\bibitem[HX13]{HX13}
	C.~D. Hacon and C.~Xu, \emph{Existence of log canonical closures}, Invent.
	Math. \textbf{192} (2013), no.~1, 161--195.

    \bibitem[IP99]{fano type}
	V.~A.~Iskovskikh and Y.~G.~Prokhorov, \emph{{Fano varieties}}, in Algebraic Geometry, V, Encycl. Math. Sci. \textbf{47}, Springer-Verlag, Berlin, 1999, pp.~1--247.
	
	\bibitem[JLX22]{JLX22}
	J.~Jiao, J.~Liu, and L.~Xie, \emph{On generalized lc pairs with b-log abundant
		nef part}, arXiv:2202.11256 (2022).

    \bibitem[Kaw98]{kawamataadjunctionformula}
	Yujiro Kawamata, \emph{Subadjunction of log canonical divisors. {II}}, Amer. J.
	Math. \textbf{120} (1998), no.~5, 893--899.
	
	
	\bibitem[Kol93]{effectivebasepointfree}
	J.~Koll{\'a}r, \emph{Effective base point freeness}, Math. Ann. \textbf{296}
	(1993), no.~4, 595--605.
	
	\bibitem[Kol23]{familiesofgeneraltype}
	J.~Koll{\'a}r, \emph{Families of varieties of general type}, Cambridge University
	Press, 2023.
	
	\bibitem[Li24a]{Iitakavolume}
	Z.~Li, \emph{Boundedness of the base varieties of certain fibrations}, J. Lond.
	Math. Soc. (2) \textbf{109} (2024), no.~2, e12871.
	
	\bibitem[Li24b]{effectiveadjunctionconjecture}
	Z.~Li, \emph{A variant of the effective adjunction conjecture with
		applications}, J. Pure Appl. Algebra \textbf{228} (2024), no.~6, 107626, 22.

    \bibitem[LX23]{LX23}
	J.~Liu and L.~Xie, \emph{{Semi-ampleness of NQC generalized log canonical pairs}}, Adv. Math. \textbf{427} (2023), 109126.    

	
	
	\bibitem[Zhu25]{Zhu25}
	M.~Zhu, \emph{{Boundedness of stable minimal models with klt singularities}}, Int. Math. Res. Not. \textbf{2025} (2025), no.~2, rnae293.
	
	
	
	
	
\end{thebibliography}
\end{document}